\documentclass[a4paper,12pt,reqno]{amsart}
\usepackage{amsmath}
\usepackage{esint,graphics}
\usepackage[all]{xy}

\usepackage{a4wide}
\usepackage{amssymb}
\usepackage{amsthm}
\usepackage{amscd}
\usepackage{color}

\theoremstyle{plain}
\begingroup
\newtheorem*{theoremA}{Theorem A}
\newtheorem*{theoremB}{Theorem B}
\newtheorem*{theoremB'}{Theorem B'}
\newtheorem*{theoremB''}{Theorem B''}
\newtheorem*{theoremC}{Theorem C}
\newtheorem*{theoremD}{Theorem D}
\newtheorem{theorem}{Theorem}[section]

\newtheorem{lemma}[theorem]{Lemma}
\newtheorem{proposition}[theorem]{Proposition}
\newtheorem{corollary}[theorem]{Corollary}

\newtheorem{definition}[theorem]{Definition}
\newtheorem{rmk}[theorem]{Remark}

\newtheorem{question}[theorem]{Question}
\newtheorem{open}[theorem]{Open Problem}
\endgroup

\theoremstyle{remark}
\begingroup
\endgroup

\numberwithin{equation}{section}

\newcommand{\op}[1]{{\rm{#1}}}

\title[Global Gauges and Global Extensions]{Global gauges and global extensions in optimal spaces}

\author[M. Petrache, T. Rivi\`ere]{Mircea Petrache, Tristan Rivi\`ere}
\begin{document}
\begin{abstract}
We consider the problem of extending functions $\phi:\mathbb S^n\to \mathbb S^n$ to functions $u:B^{n+1}\to \mathbb S^n$ for $n=2,3$. We assume $\phi$ to belong to the critical space $W^{1,n}$ and we construct a $W^{1,(n+1,\infty)}$-controlled extension $u$. The Lorentz-Sobolev space $W^{1,(n+1,\infty)}$ is optimal for such controlled extension. Then we use such results to construct global controlled gauges for $L^4$-connections over trivial $SU(2)$-bundles in $4$ dimensions. This result is a global version of the local Sobolev control of connections obtained by K. Uhlenbeck \cite{Uhl2}.
\end{abstract}

\maketitle
\tableofcontents
\section{Introduction}
\noindent
The use of Hodge decomposition is by now one of the classical tools in the study of elliptic systems and is related to important breakthroughs such as the famous ``$\op{div}$-$\op{curl}$''-type theorems \cite{CLMS}. More recently such decomposition has allowed to solve \cite{rivconfinv} S. Hildebrandt's conjecture \cite{hildebrandt}, and at the same time establishing an important link to an apparently unrelated fields of geometry, such as the study of conformally invariant geometric problems in $2$-dimensions \cite{helein} and the study of Yang-Mills bundles and gauge theory \cite{Uhl2}, with the introduction of controlled Coulomb gauges.\\

The study of $2$-dimensional problems using controlled gauges has already given its fruits, and in connection to the discovery of H. Wente's inequality (which gave the basis for introducing the Lorentz spaces  $L^{(2,\infty)}$ in geometric problems) allowed the successful use of controlled moving frames in the study of harmonic maps and prescribed mean curvature surfaces \cite{helein}, \cite{mullsver}. We come back to this in Section \ref{ssec:helein}. Techniques and function spaces related to the moving frame method also apply to the study of the Willmore functional \cite{rivcourse} for immersed surfaces.\\

The use of controlled gauges especially in relation to Lorentz spaces in dimensions higher than $2$ is far less developed. We attempted here a first attack of this completely new area of research, and we obtained some extensions of previous results for the case of Yang-Mills fields on $4$ dimensional manifolds. 

\subsection{Yang-Mills theory and controlled gauges}
Yang-Mills theory for $4$-manifolds is often associated to the famous result of S. Donaldson \cite{donaldson} who, using the moduli spaces of anti-selfdual connections, described new invariants of smooth manifolds.\\

The study of moduli spaces used by Donaldson \cite{donaldson} starts from the result of K. Uhlenbeck \cite{Uhl2}, who proved that one can find a gauge in which the $W^{1,2}$-norm of the local coordinate expression of the connection is controlled by the $L^2$-norm of the curvature. Moreover the connection $1$-form $A$ can be also made to satisfy the Coulomb condition $d^*A=0$.\\

It is easy to construct a Coulomb gauge in which we have just an $L^2$-control in terms of the curvature (cfr. \cite{petrachethesis} or \cite{petriv}). This is done by first obtaining any gauge in which 
\[
 \|A\|_{L^2}\leq C \|F\|_{L^2}
\]
and then finding the smallest norm coefficients with respect to that gauge on our manifold $M$:
\[
\min\left\{\int_M |g^{-1}dg + g^{-1}Ag|^2 dx\::\:g\in W^{1,2}(M, SU(2))\right\}.
\]
A unique minimizer will exist by convexity and it will satisfy the Coulomb equation $d^*A=0$.\\

The control of $A$ in the higher norm $W^{1,2}$ is more difficult. A smallness hypothesis on $\|F\|_{L^2(M)}$ must be required in order for the control to be achievable:
\begin{theorem}[controlled Coulomb gauge under assumption of small energy, \cite{Uhl2}]\label{uhlcoulsmall}
There exists a constant $\epsilon_0>0$ such that if the curvature satisfies $\int_{M}|F|^2\leq\epsilon_0$ then there exists a Coulomb gauge $\phi\in W^{2,2}(M,SU(2))$ such that in that gauge the connection satisfies $\|A_\phi\|_{W^{1,2}(M)}\leq C\|F\|_{L^2(M)}$ with $C>0$ depending only on the dimension.
\end{theorem}
The reason why the smallness of the curvature is necessary is that $\|F\|_{L^2(M)}$ being above a certain threshold allows the second Chern number of the bundle to be nontrivial:
\[
c_2(E)=\frac{1}{8\pi^2}\int_{M}\op{tr}(F\wedge F)\neq 0.
\]
If for such $F$ the controlled gauge would be \emph{global}, i.e. if we would have a global trivialization in which the connection of the above $F$ is expressed as $d+A$ with
\[
 \|A\|_{W^{1,2}(M)}\leq C,
\]
then by Sobolev and H\"older inequalities we would have enough control on the quantities involved to prove the following formal identity for our $A$:
\[
\op{tr}\left[(dA+[A,A])\wedge(dA+[A,A])\right]=d\;\op{tr}\left(A\wedge dA+\frac{2}{3}A\wedge A\wedge A\right).
\]
Now the right side is an exact form, thus it has integral equal to zero over the boundaryless manifold $M$, contradicting $c_2(E)\neq 0$.\\

M. Atiyah-N. Hitchin-I. Singer \cite{ahs} and C. Taubes \cite{taubes} constructed instantons with nontrivial Chern numbers as in the above heuristic. To exemplify the phenomena at work consider the simplest instanton, having $c_2(E)=1$ over $M=\mathbb S^4$ (cfr. \cite{freeduhl}, Ch. 6 for notations and details). Recall that we may use quaternion notation due to the isomorphisms $SU(2)\sim Sp(1)$ and $su(2)\sim Im\mathbb H$, under which Pauli matrices correspond to quaternion imaginary units. We then have the following local expression of $A$ over $\mathbb R^4$ (identified by stereographic projection with $\mathbb S^4\setminus\{p\}$) in a trivialization:
\[
A=Im\left(\frac{x\;d\bar x}{1+|x|^2}\right).
\]
If $\Psi$ is the inverse stereographic projection then $\Psi^*A$ is smooth away from the pole $p$, but near $p$ we have $|\Psi^*A|(q)\sim \op{dist}_{\mathbb S^4}(p,q)^{-1}$, which is not $L^4$ in any neighborhood of $p$.\\

Such behavior like $\frac{1}{|x|}$ shows that we are in any space $L^p$ for $p<4$ but not in $L^4$. The natural space is the weak-$L^4$ space, which is strictly contained between all $L^p,p<4$ and $L^4$:
\begin{definition}[see \cite{grafakos}]
Let $X,\mu$ be a measure space. The space $L^{p,\infty}(X,\mu)$ (also called \emph{weak-$L^p$} or \emph{Marcinkiewicz} space) is the space of all measurable functions $f$ such that 
\[
 \|f\|_{L^{p,\infty}}^p:=\sup_{\lambda>0}\lambda^p\mu\{x: |f(x)|>\lambda\}
\]
is finite.
\end{definition}
We note immediately that the function $f(x)=\frac{1}{|x|}$ belongs to $L^{4,\infty}$ on $\mathbb R^4$ and the above global gauge gives an $L^{4,\infty}$ $1$-form $\Psi^*A$ on $\mathbb S^4$. Spaces $L^{p,\infty}$ arise naturally in dealing to the critical exponent estimates for elliptic equations. The Green kernel $K_n(x)$ of the Laplacian on $\mathbb R^n$ satisfies indeed $\nabla K\in L^{\frac{n}{n-1},\infty}$ but not $\nabla K\in L^{\frac{n}{n-1}}$. Thus $\Delta u=f$ with $f\in L^1$ implies $\nabla u=\nabla K*f\in L^{\frac{n}{n-1},\infty}$ by an extended Young inequality (see \cite{grafakos}), unlike the higher exponent case $f\in L^p, p>1$, which gives the stronger result $\nabla u\in L^p$.\\

\subsection{Controlled global gauges}
As shown heuristically by the explicit case of the instanton $A$ above, it is known how to construct $L^{4,\infty}$ global gauges. Our main effort in this work is to obtain a \underline{norm-controlled} gauges, mirroring Theorem \ref{uhlcoulsmall} by K. Uhlenbeck. The main result is the following:

\begin{theoremA}
Let $M^4$ be a Riemannian $4$-manifold. There exists a function $f:\mathbb R^+\to\mathbb R^+$ with the following properties.\\
 Let $\nabla$ be a $W^{1,2}$ connection over an $SU(2)$-bundle over $M$. Then there exists a \underbar{global}
$W^{1,(4,\infty)}$ section of the bundle (possibly allowing singularities) over the whole $M^4$ such that in the corresponding trivialization $\nabla$ is given by $d+A$ with the following bound.
\begin{equation*}
\|A\|_{L^{(4,\infty)}}\leq f\left(\|F\|_{L^2(M)}\right),
\end{equation*}
where $F$ is the curvature form of $\nabla$.
\end{theoremA}
This theorem is related to a second main result of this work, namely the introduction of Lorentz-Sobolev extension theorems for nonlinear maps. This result takes most of our efforts and can be stated as follows:
\begin{theoremB}
There exists a function $f_1:\mathbb R^+\to\mathbb R^+$ with the following property. Suppose $\phi\in W^{1,3}(\mathbb S^3,\mathbb S^3)$. then there exists an extension $u\in W^{1,(4,\infty)}(B^4,\mathbb S^3)$ of $\phi$ such that the following estimate holds:
\begin{equation*}
 \|\nabla u\|_{L^{4,\infty}(B^4)}\leq f_1\left(\|\nabla\phi\|_{L^3}\right).
\end{equation*}
\end{theoremB}
The originality of Theorem B with respect to the previous results \cite{bethueldemengel} or \cite{mucci} is that whereas the previous works were concerned with the \underline{existence} of an extension, in our case a \underline{control} is provided in term of the boundary value. We show below that even under the hypothesis $\op{deg}(\phi)=0$ such that a $W^{1,4}$-extension surely exists, no energy control will be available.\\

Controlled global gauges as above will probably have many applications in the analysis of gauge theory, as for example in simplifying compactness results; see \cite{petrachethesis}. Controlled global gauges could allow a global control on the Yang-Mills flow, provided we obtain also the Coulomb condition, which is however an open question:
\begin{open}
 Prove that it is possible to find $L^{4,\infty}$-controlled global Coulomb gauges as in Theorem A. In other words, prove that it is possible to find a gauge as in Theorem A, but with the further requirement that $d^*A=0$ in such gauge.
\end{open}

\subsection{Strategy of gauge construction} 
The link between Theorems A and B is given by the well-known identification $SU(2)\simeq \mathbb S^3$. Therefore Theorem B can be rephrased as follows:
\begin{theoremB'}
 Fix a trivial $SU(2)$-bundle $E$ over the ball $B^4$. There exists a function $f_1:\mathbb R^+\to\mathbb R^+$ with the following property. If $g\in W^{1,3}(\mathbb S^3,SU(2))$ gives a trivialization of the restricted bundle $E|_{\partial B^4}$, then there exists an extension of $g$ to a trivialization $\tilde g\in W^{1,(4,\infty)}(B^4,SU(2))$ such that the following estimate holds:
\begin{equation*}
 \|\nabla \tilde g\|_{L^{4,\infty}(B^4)}\leq f_1\left(\|\nabla g\|_{L^3(\mathbb S^3)}\right).
\end{equation*}
\end{theoremB'}
The proof of the Theorem A is by a sequence of gauge extensions along the simplexes of a suitable triangulation. We use simplexes where Uhlenbeck's result \ref{uhlcoulsmall} holds, i.e. $F$ has energy $\lesssim \epsilon_0$. To ensure a lower bound on the size of simplexes we cut areas of energy concentration and use induction on the energy, see the summary \eqref{summary1}.

\subsection{Extension of Sobolev maps into manifolds}\label{ssec:extintro}
We discuss the relevance of our theorem, several possible extensions and related phenomena in Section \ref{sec:survey}.\\

Here we point out the main open questions in the area of controlled nonlinear extensions and some analogues of Theorem B. An useful tool to control the topology of a manifold $N$ are the fundamental groups $\pi_m(N)$ which is a quotient of $C^0(\mathbb S^m, \mathbb N)$. To say that any map in this space is continuously extendable to $B^{m+1}$ amounts to asserting that $\pi_m(N)=0$.\\

We consider here the \emph{controlled} extension problem for maps $\mathbb S^m\to \mathbb S^n$. As is usually the case the interesting new features appear when smooth maps are not dense in $W^{1,p}(S^m,S^n)$, in which case we expect topological obstructions to gradually disappear as $p$ decreases. The first facts to note are the following:
\begin{itemize}
 \item For extensions of maps from $W^{1,p}(\mathbb S^m,\mathbb S^n)$ to $B^{m+1}$ the natural space given by continuous Sobolev and trace embeddings is $W^{1,\frac{m+1}{m}p}(B^{m+1}, \mathbb S^n)$ (see Sec. \ref{ssec:hardtlintrick} and \ref{ssec:largeintext}).
 \item For $p<\frac{n+1}{m+1}m$ the controlled extensions exist (see Sec. \ref{ssec:hardtlintrick}).
 \item $p>m$ the extension question reduces to a purely topological problem (see Sec. \ref{ssec:largeintext}).
\end{itemize}
The open cases when $p<m$ are the thus among the following ones:
\begin{open}\label{highdimdomain}
 Assume that $\frac{n+1}{m+1}m\leq p<m$ and $m>n$. For which such choices of $m,n,p$ does there exist a finite function $f_{m,n,p}:\mathbb R^+\to\mathbb R^+$ such that for every $\phi\in W^{1,p}(\mathbb S^m, \mathbb S^n)$ there exists an extension $u\in W^{1,\frac{m+1}{m}p}(B^{m+1}, \mathbb S^n)$ for which the estimate 
 \[
  \|u\|_{W^{1,\frac{m+1}{m}p}(B^{m+1}, \mathbb S^n)}\leq f_{m,n,p}\left(\|\phi\|_{W^{1,p}(\mathbb S^m, \mathbb S^n)}\right)
 \]
holds? Does the estimate hold for $p=m$ for the norm $W^{1,(m+1,\infty)}(B^{m+1},S^n)$?
\end{open}
The above Open Problem is partially understood or solved just in some cases:
\begin{itemize}
 \item Due to a relation of extension problems to lifting problems we answer the above problem for $n=2<m$ and $\frac{3m}{m+1}\leq p<\frac{4m}{m+1}$, see Prop. \ref{sms2} and Section \ref{ssec:lifting}.
  \item  In particular we cover all $p$ for the dimensions $m=3, n=2$. 
 \item For $n=1, m\geq 3$ and $\frac{3m}{m+1}\leq p<m$ \cite{bethueldemengel} prove that no extension exists.
\end{itemize}

It will be interesting in the future to look at the link of extension and lifting problems in detail. It is possible to do this also in the case of $\mathbb S^1$-valued maps and in nonlocal Sobolev spaces, e.g. using the results of \cite{BBM}.\\

In the critical case $p=m$ left aside in the above Open Problem we have the following results:
\begin{itemize}
\item Using the Hopf lifts as in \cite{hr1, hr} we prove Theorem C which is the solution to case $p=m=n=2$ (see Sec. \ref{sec:hopflift}).
 \item The extension exists but cannot be controlled in the above Sobolev norm, making the Lorentz-Sobolev weakening of Theorem B and of Theorem C below optimal (see Sec. \ref{ssec:smallenw14ext}). This is analogous to the case of global gauges in $4$-dimensions pointed out in the introduction.
 \item We also prove an analogous result for $p=m=n=1$ (see Theorem \ref{1dcase}) however this is not the natural space to look at, unlike higher dimensions. In this case indeed the trace space $H^{1/2}(\mathbb S^1,\mathbb S^1)$ is the natural space to look at, because $W^{1,1}(\mathbb S^1,\mathbb S^1)$ does not continuously embed in it (we recall a counterexample in \ref{ssec:1dext}).
\end{itemize}
  These theorems leave open higher dimensional cases:
\begin{open}
Assume $n\geq 4$. Does there exist a finite function $f_n:\mathbb R^+\to\mathbb R^+$ such that for each $\phi\in W^{1,n}(\mathbb S^n,\mathbb S^n)$ we can find an extension $u\in W^{1,(n+1,\infty)}(B^{n+1},\mathbb S^n)$ for which the estimate
\[
 \|u\|_{W^{1,(n+1,\infty)}(B^{n+1},\mathbb S^n)}\leq f_n\left(\|\phi\|_{W^{1,n}(\mathbb S^n,\mathbb S^n)}\right)
\]
holds?
\end{open}
Unlike linear Sobolev spaces, not only the topology of the domain must be compared to the Sobolev exponent $p$, but also the dimension and structure of the constraint (i.e. the target manifold) plays a critical role. This is also related to the topological global obstructions to density results for smooth functions between manifolds in F. Hang-F. Lin \cite{hanglin, hl2} (see also T. Isobe \cite{isobeglobalobstr}).\\

A general tool allowing extensions the projection trick of Section \ref{ssec:hardtlintrick}, which works well for Sobolev exponents smaller than the target dimension plus one. Lifting theorems allow to increase this dimension thus to apply the projection trick with higher exponents.\\

Using the Hopf fibration $H:\mathbb S^3\to\mathbb S^2$ we construct controlled lifts and apply a version of the projection trick obtaining the following theorem with much less effort than for the $3$-dimensional case of Theorem B:
\begin{theoremC}[see Section \ref{sec:hopflift}]
  Suppose $\phi\in W^{1,2}(\mathbb S^2, \mathbb S^2)$ is given. Then there exists $u\in W^{1,(3,\infty)}(B^3, \mathbb S^2)$ such that in the sense of traces $u|_{\partial B^3}=\phi$ and such that the following estimate holds, for a constant independent of $\phi$.
\begin{equation*}
 \|u\|_{W^{1,(3,\infty)}(B^3)}\leq C\|\phi\|_{W^{1,2}(\mathbb S^2)}(1+\|\phi\|_{W^{1,2}(\mathbb S^2)}).
\end{equation*}
\end{theoremC}

The Hopf fibration has a natural structure of $U(1)$-bundle with nontrivial characteristic class, $P\to S^2$. Lifting a map $\phi:X\to S^2$ to a $\tilde\phi:X\to S^3$ for which $H\circ\tilde\phi=\phi$ corresponds to giving the trivialization of the pullback bundle $\phi^*P$. Analogous lifts are interesting to study for general principal $G$-bundles, 
using universal connections. The next case after the one with target $S^2$ is the $SU(2)$-bundle of the introduction, which corresponds to the Hopf fibration $\mathbb S^7\to \mathbb S^4$.\\

The Hopf lift seems to be much more difficult to extend the case where the target is $\mathbb S^3$. We cannot use principal bundles because $\pi_2(G)=0$ for all compact Lie groups $G$. For other fibrations the following question is open:
\begin{open}\label{opnolift}
 Is it possible to find a fibration $\pi:E\to \mathbb S^3$ with compact fiber $M$ and a constant $C>0$ such that for each $\phi\in W^{1,3}(\mathbb R^3,\mathbb S^3)$ there exists a lift $\tilde\phi:\mathbb R^3\to E$ satisfying the estimate $\|\nabla\tilde\phi\|_{L^{(3,\infty)}}\leq C f(\|\nabla\phi\|_{L^3})$ for some finite function $f:\mathbb R^+\to\mathbb R^+$?
 \end{open}

The controlled Hopf lift result for $S^2$ yields also an answer to Open Problem \ref{highdimdomain} for dimensions $m=3, n=2$:
\begin{theoremD}\label{s3s2}
 Assume $\phi\in W^{1,3}(\mathbb S^3,\mathbb S^2)$. Then there exists a controlled extension $u\in W^{1,(4,\infty)}(B^4,\mathbb S^2)$ with the control
 \[
  \|u\|_{W^{1,(4,\infty)}(B^4,\mathbb S^2)}\leq C\|\phi\|_{W^{1,3}(\mathbb S^3,\mathbb S^2)}(1+\|\phi\|_{W^{1,3}(\mathbb S^3,\mathbb S^2)}).
 \]
If instead we have $\phi\in W^{1,p}(\mathbb S^3,\mathbb S^2)$ for $9/4\leq p<3$ then there exists an extension $u\in W^{1,\frac{4}{3}p}(B^4,\mathbb S^2)$ with
\[
  \|u\|_{W^{1,\frac{4}{3}p}(B^4,\mathbb S^2)}\leq C\|\phi\|_{W^{1,p}(\mathbb S^3,\mathbb S^2)}(1+\|\phi\|_{W^{1,p}(\mathbb S^3,\mathbb S^2)}).
 \]
\end{theoremD}

The same proof allows to also answer Open Problem \ref{opnolift} for $n=2<m$ for some exponents $p$:
\begin{proposition}\label{sms2}
Assume $n=2, m\geq 3$ and $\frac{3m}{m+1}\leq p<\frac{4m}{m+1}$ and consider a $\phi\in W^{1,p}(\mathbb S^m,\mathbb S^2)$. Then there exists a controlled extension $u\in W^{1,\frac{m+1}{m}p}(B^{m+1},\mathbb S^2)$ with
\[
  \|u\|_{W^{1,\frac{m+1}{m}p}(B^4,\mathbb S^2)}\leq C\|\phi\|_{W^{1,p}(\mathbb S^3,\mathbb S^2)}(1+\|\phi\|_{W^{1,p}(\mathbb S^3,\mathbb S^2)}).
 \]
\end{proposition}

\subsection{Ingredients used in the construction of $W^{1,(4,\infty)}(B^4,\mathbb S^3)$ extensions}
The starting new idea was to the use of implicit function theorems and of a limit on the integrability exponent as done in \cite{Uhl1} for extension result. The procedure of Appendix \ref{sec:Uhlenbeck} is generalizable to other contexts with no new ingredients, at least as long as a Lie group structure is present.\\

For the implicit function theorems above we needed here a new product estimate valid in Sobolev spaces, which is presented in Appendix \ref{sec:product}, extending partially the results of \cite{brezismir}, cfr. \cite{runstsickel} and \cite{triebel}.\\

The second idea was to use $L^{(4,\infty)}$ functions such that the $L^4$-estimate would fail just near a controlled number of points. Such singular points (where ``singular'' is meant with respect to the $L^4$ estimates) are introduced via Lemma \ref{intest} and Theorem \ref{smallcest}.\\

The uniform $L^{(4,\infty)}$-control is obtainable just in the case where the boundary value has no large energy ``hot spots''. To deal with the case where energy concentrates we use two tools which are available in the particular case of $\mathbb S^3\simeq SU(2)$: (1) the group operation of $SU(2)$, which gives a continuous product on $W^{1,3}(X, \mathbb S^3)$; (2) the M\"obius group of $\mathbb S^3$, coupled with the conformal invariance of the $L^3$-norm of the gradient on $\mathbb S^3$.\\

Under a balancing condition on the boundary value $\phi$ we can write $\phi=\phi_1\phi_2$ where the product is taken in $SU(2)$, and the energies of $\phi_i, i=1,2$ are strictly less than that of $\phi$, allowing an induction on the energy. If the balancing is not valid, we apply a M\"obius transformation $F_v$ to $\mathbb S^3$ and either reduce to a balanced situation for $F_v\circ \phi$ and for some $v$ or provide a substitute $v\in B^4\mapsto \int_\mathbb S^3\phi\circ F_v$ to the harmonic extension of $\phi$, to which we can now apply the projection trick. The natural parameterization of the M\"obius group of $\mathbb S^3$ via vectors in $B^4$ fits very well in this setting, and we were inspired to use it by the similar use of it in \cite{marquesneves}.

\subsection{Plan of the paper}
Section \ref{sec:survey} contains a list of positive and negative results concerning phenomena parallel to ours, showing that our results are optimal. Section \ref{sec:hopflift} contains the proof of Theorem C. In Section \ref{sec:extw14} we prove Theorem B and in Section \ref{sec:controlledcoul} we prove the Theorem A. Appendix \ref{sec:Uhlenbeck} deals with our new ``extension'' version of Uhlenbeck's gauge construction and in Appendix \ref{sec:product} we prove the needed new product inequality. Appendix \ref{sec:moeblemmas} contains computations and notation for the M\"obius groups of $B^4$ and $\mathbb S^3$.

\section{Controlled and uncontrolled nonlinear Sobolev extensions}\label{sec:survey}
\noindent
Classical Sobolev Space theory features optimal extension theorems in natural trace norms. For example if $\Omega\subset \mathbb R^n$ is a bounded smooth domain and $u:\partial\Omega\to \mathbb R$ is a $W^{1,n-1}$-function then there exists an extension $\bar u:\Omega\to\mathbb R$ such that $\bar u \in W^{1,n}$ and the estimate 
$$
\|\bar u\|_{W^{1,n}}\leq C\|u\|_{W^{1,n-1}}
$$
holds (with $C$ independent of $u$). This extension theorem is optimal in the sense that for dimensions $n>2$ the natural trace operator $\bar u\in W^{1,n}(\Omega)\mapsto \bar u|_{\partial \Omega}$ sends $W^{1,n}$ to the optimal space $W^{1-\frac{1}{n},n}$ (see \cite{tartar} chapter 40 for the natural appearance of this space), and we have the optimal Sobolev continuous embedding $W^{1-\frac{1}{n},n}\to W^{1,n-1}$ (see \cite{tartar}) which brings us back to the original space. A similar result still holds if we replace the codomain $\mathbb R$ by $\mathbb R^m$.\\

However for $n=2$ the space $W^{1,1}(\mathbb S^1,\mathbb S^1)$ does not continuously embed in $H^{1/2}(\mathbb S^1,\mathbb S^1)$, making the above reasoning less poignant, see Sec. \ref{ssec:1dext}.\\

A possible construction of $\bar u$ can be done by imitating the following model valid for $\Omega=\mathbb R^n_+:=\{(x_1,\ldots,x_n)| x_n\geq 0\}$: 
$$
\bar u(x_1,\ldots,x_{n-1},\epsilon):=(\rho_\epsilon * u)(x_1,\ldots, x_{n-1}),
$$
where $\rho_\epsilon$ is a usual family of radial smooth compactly supported mollifiers.\\

An equivalent construction of $\bar u$ in terms of function spaces is by harmonic extension. The optimal result is the following
\begin{proposition}[harmonic extension, cfr. \cite{GGS} Ch. 10]\label{harmext}
 Assume $q>1$ and $u\in W^{1-\frac{1}{q},q}(\partial B^{m+1},\mathbb R^{n+1})$. Then there exists a harmonic extension $\bar u\in W^{1,q}(B^{m+1},\mathbb R^{n+1})$ such that
 \[
  \|\bar u\|_{W^{1,q}(B^{m+1},\mathbb R^{n+1})}\leq C_{m,n,q}\|u\|_{W^{1-\frac{1}{q},q}(\partial B^{m+1},\mathbb R^{n+1})}.
 \]
\end{proposition}

By Sobolev embedding we have the controlled inclusion $W^{1,p}\hookrightarrow W^{1-\frac{1}{q},q}$ on an $m$-dimensional bounded open domain (or a compact manifold like $\partial B^{m+1}$) for $q\leq\frac{m+1}{m}p$, therefore this $q$ is the largest exponent where we can hope to have a control for the extension.\\

If $u$ is a constrained function with values in a subset of $\mathbb R^{n+1}$ (e.g. a curved $n$-dimensional submanifold like $\mathbb S^n$) then averaging even on a very small scale could push the values of $\bar u$ quite far from the constraint obeyed by $u$. This happens in particular for Sobolev exponents making the dimension ``supercritical'', i.e. exponents such that $W^{1,q}(B^{m+1})$ is not constituted of continuous functions. We pass to describe some cases where directly projecting back to $\mathbb S^n$ does not destroy the norm control of Prop. \ref{harmext}(harmonic extension).

\subsection{Projection from a well-chosen center}\label{ssec:hardtlintrick}
We present in this section a trick which probably appeared for the first time in relation to nonlinear Sobolev extensions in \cite{hardtkindlin} and \cite{hardtlin}. For a Lorentz space version cfr. Prop. \ref{hltrick}(projection trick 2).

\begin{proposition}[projection trick]\label{trickhkl}
If $f\in W^{1,q}(\Omega,B^{n+1})$ with $q<n+1$ and $\Omega$ is a bounded open simply connected domain of $\mathbb R^{m+1}$ then there exists $a\in B_{1/2}^{n+1}$ and a constant $C$ depending only on $q,m,n$ such that if $f_a(x)=\pi_a(f(x))$ where $\pi_a:B^{n+1}\setminus \{a\}\to \mathbb S^n$ is the projection which is constant along the segments $[a,\omega],\omega\in \mathbb S^n$, then
$$
\|f_a\|_{W^{1,q}(\Omega, \mathbb S^n)}\leq C\|f\|_{W^{1,q}(\Omega, B^{n+1})}.
$$
\end{proposition}
\begin{proof}
 We have just to estimate the gradient of $f_a$ in terms of that of $f$ since the functions themselves are anyways bounded and $\Omega$ is assumed of finite measure. We first note that since $a\in B^{n+1}_{1/2}$ is away from the boundary of $B^{n+1}$, we have the pointwise estimate
$$
|\nabla f_a|(x)\lesssim\frac{|\nabla f|(x)}{|f(x)-a|},
$$
where the implicit constant depends only on $n$. We next consider the following ``average'' on $a$:
$$
\int_{B_{1/2}^{n+1}} \left(\int_{\Omega}|\nabla f_a|^q(x)dx\right)da\lesssim\int_\Omega|\nabla f|^q(x)\left(\int_{B^{n+1}_{1/2}}\frac{da}{|f(x)-a|^q}\right)dx.
$$
We note that the inner integral is of the form 
$$
I(y):=\int_{B_{1/2}^{n+1}}\frac{da}{|y-a|^q},
$$
and 
$$
\max_y I(y)=I(0)=C_n\int_0^{1/2}r^{n+q}dr=C_{n,q}<\infty \text{ since }q<n+1.
$$
therefore we obtain 
$$
\int_{B_{1/2}^{n+1}}\|\nabla f_a\|_{L^q}^qda\leq C_{n,q}\|\nabla f\|_{L^q}^q,
$$
and the proof is easily concluded.
\end{proof}
The above proposition together with Prop. \ref{harmext}(harmonic extension) and the remark on Sobolev exponents following it gives the following:
\begin{theorem}[corollary of the projection trick, cfr \cite{hardtlin} Thm. 6.2]\label{extsmallp}
 Let $m,n\in\mathbb N^*$. If $1\leq p<\frac{n+1}{m+1}m$ then for any $\phi\in W^{1,p}(\partial B^{m+1},\mathbb S^n)$ there exists a nonlinear extension $u\in W^{1,\frac{m+1}{m}p}(B^{m+1},\mathbb S^n)$ satisfying the control
 \[
  \|u\|_{W^{1,\frac{m+1}{m}p}(B^{m+1},\mathbb S^n)}\leq C_{m,n,p}\|\phi\|_{W^{1,p}(\partial B^{m+1},\mathbb S^n)}.
 \]
\end{theorem}
\begin{rmk}
 Note that from the same ingredients we obtain also the stronger estimate where for $q:=\frac{m+1}{m}p<m$ the weaker space $W^{1-\frac{1}{q},q}(\partial B^{m+1},\mathbb S^n)$ replaces $W^{1,p}(\partial B^{m+1},\mathbb S^n)$. This was done in \cite{bethueldemengel} and \cite{hardtlin}. We stated Theorem \ref{extsmallp} as above to emphasize the connection with our Theorems B and C. Indeed taking $m=n$ we see that those Theorems cover the critical exponent $p=n$, for which the projection trick stops working.
\end{rmk}
\subsection{Large integrability exponents}\label{ssec:largeintext}
We now consider functions in $W^{1,p}(\mathbb S^m,\mathbb S^n)$ with $p>m$. The space $C^{0,1-m/p}(\mathbb S^m,\mathbb S^n)$ continuously embeds in this space. The candidate extension space $W^{1,\frac{m+1}{m}p}(B^{m+1},\mathbb S^n)$ is made of $C^{0,1-m/p}$-functions as well. As described in Sec. \ref{ssec:extintro}, extension problem is guaranteed to have a solution as long as $\pi_m(\mathbb S^n)=0$. This is true for $m<n$ but false for many choices of $m>n$ and for $m=n$.\\

When an extension exists i.e. for $\phi$ representing the identity of $\pi_m(\mathbb S^n)\neq 0$, a controlled extension can be constructed, based on the fact that a bound on the $C^{0,\alpha}$-norm for $\alpha>0$ implies a control on the modulus of continuity.

\subsection{Extension for maps in $W^{1,1}(\partial \mathbb S^1,\mathbb S^1)$}\label{ssec:1dext}
For maps with values in $\mathbb S^3$ we are helped by the existence of a well-behaved product structure on $\mathbb S^3$, i.e. the one which gives the identification $\mathbb S^3\simeq SU(2)$. This is enough to get the analogous result for $n=1$ as we will see now. It is however well-known (see \cite{hatcherVB} 2.3) that this is a very unusual case: a group operation exists on $\mathbb S^k$ only for $k=1,3$.\\

We can state a similar extension problem in the $1$-dimensional case. This kind of controlled extension result is related to the recent work on Ginzburg-Landau functionals in \cite{serftice}.\\

Here the main structural ingredients present for $\mathbb S^3$ are again present: namely, we have a group operation on $\mathbb S^1$ (in this case it is even an abelian group) and a M\"obius structure on $D^2$, restricting to one on $\mathbb S^1$. We follow the strategy of proof described in Sec. \ref{ssec:extintro}. The result is:
\begin{theorem}[$1$-dimensional version of the extension]\label{1dcase}
 There exists a function $g:\mathbb R^+\to\mathbb R^+$ with the following property. If $\phi\in W^{1,1}(\mathbb S^1,\mathbb S^1)$ then there exists $u\in W^{1,(2,\infty)}(D^2,\mathbb S^1)$ with $u|_{\partial D^2}=\phi$ in the sense of traces and we have the norm control
$$
\|u\|_{W^{1,(2,\infty)}(D^2,\mathbb S^1)}\leq g(\|\phi\|_{W^{1,1}(\mathbb S^1,\mathbb S^1)}).
$$
\end{theorem}
We will explain the changes which occur with respect to the proof of Theorem B (see Sec. \ref{sec:extw14}).
\begin{proof}[Sketch of proof:]
 The procedure is as in Section \ref{sec:extw14} and Appendix \ref{sec:Uhlenbeck}, we have just to replace exponents and dimensions $3,4$ with $1,2$. For the analogue of Proposition \ref{prop:moebius}(balancing $\Rightarrow$ extension) the biharmonic equation \eqref{eq:biharm} is replaced by a harmonic equation, while the resulting estimates persist. Perhaps the only main change is Lemma \ref{triebelest} of Section \ref{sec:product} changes more drastically. It should be replaced by the following product estimate valid for $f\in W^{1,1}(D^2), g\in L^\infty\cap W^{1,2}(D^2)$:
$$
\|fg\|_{W^{1,1}}\leq\|f\|_{W^{1,1}}\left(\|g\|_{L^\infty} +\|g\|_{W^{1,2}}\right)
$$
\end{proof}
We must however note that the naturality of the space $W^{1,1}(\mathbb S^1, \mathbb S^1)$ in Theorem \ref{1dcase} is less evident, since the trace space $H^{1/2}(\mathbb S^1,\mathbb S^1)$ does not continuously embed in it, unlike what happens in higher dimensions. This is seen by considering 
\[
u_\epsilon(\theta)=\exp\left(i\;\min\left\{1, \epsilon^{-1}\op{dist}_{\mathbb S^1}(\theta,[-\pi/2,\pi/2])\right\}\right).                                                                                                                                                                                                                                                \]
It is then clear that $\|\nabla u_\epsilon\|_{L^1(\mathbb S^1)}=2$ while we estimate the double integral in $\theta,\theta'$ giving the $H^{1/2}$-norm by the contribution of the regions $\theta\in [0,\pi/2], \theta'\in[\pi/2+\epsilon,\pi+\epsilon]$. Under these choices $u_\epsilon(\theta)=e^0,u_\epsilon(\theta')=e^i$ and their distance in $\mathbb S^1$ is $1$. Thus
\begin{eqnarray*}
\|u_\epsilon\|_{H^{1/2}(\mathbb S^1,\mathbb S^1)}^2&=&\int_{\mathbb S^1}\int_{\mathbb S^1}\frac{\op{dist}_{\mathbb S^1}(u_\epsilon(\theta),  u_\epsilon(\theta'))^2}{\op{dist}_{\mathbb S^1}(\theta,\theta')^2}d\theta d\theta'\\
&\leq& \int_0^1\int_0^1\frac{1}{|x+2\epsilon/\pi - y|^2}dx \;dy\\
&\lesssim&|\log\epsilon|+1.
\end{eqnarray*}

\subsection{Using controlled liftings to obtain controlled extensions}\label{ssec:lifting}
The control obtained for extensions of maps in $W^{1,3}(\mathbb S^3,\mathbb S^3)$ and $W^{1,1}(\mathbb S^1,\mathbb S^1)$ is exponential in the norms of these maps. In Section \ref{sec:hopflift} we describe an approach working for $\phi\in W^{1,2}(\mathbb S^2,\mathbb S^2)$ which is completely different than in dimensions $1,3$ and yields a faster proof and a better control. Such approach was first considered in \cite{hr1}. This is based on the existence of controlled Hopf lifts. The result is (see Corollary \ref{corollhopf}) that there exists a $L^{2,\infty}$-controlled lifting $\tilde \phi:\mathbb S^2\to \mathbb S^3$ i.e. a function such that $H\circ\tilde \phi=\phi$  where $H:\mathbb S^3\to \mathbb S^2$ is the Hopf fibration and we have the control
\[
||\nabla \tilde \phi||_{L^{2,\infty}}\leq C||\nabla \phi||_{L^2}(1+||\nabla \phi||_{L^2}).
\]
The analogous controlled lift exists also for $\phi\in W^{1,3}(\mathbb S^3,\mathbb S^2)$, whereas for $2\leq p<3$ we have a control on the $L^p$-norm of the lift instead of the $L^{p,\infty}$ one, cfr. Prop. \ref{sms2}. This lift allows to prove, along the same lines, Theorem C and Theorem D.\\

The gist of the proof is the following. Once we have the controlled lift indeed, the lifted map takes values into a sphere of a higher dimension. This allows a wider range of application to the projection trick of Prop. \ref{trickhkl}(projection trick) or of its Lorentz space analogue of Prop. \ref{hltrick}(projection trick 2).\\

After having extended the lift, re-projecting the extension to $S^2$ via the Hopf map maintains the gradient estimates. This is due to the fact that the Hopf fibration is a submersion (cfr. \eqref{normidentity}) and our lift can be taken such that also the ``vertical'' component $\eta$ is controlled.\\

The existence of nonlinear liftings has been so far very active regarding $S^1$-valued maps (see e.g. \cite{BBM}, \cite{BZ},\cite{BBM0} and the references therein). Looking also at higher dimensional analogues seems very promising in relation to extension results.
 
\subsection{Small energy extension with estimate}\label{ssec:smallenw14ext}
As for the case of curvatures over bundles with a compact Lie group, the small energy regime allows a kind of linearization of the problem and gives estimates which are better than what expected in general. We obtain in particular an estimate in $W^{1,4}$ instead of $W^{1,(4,\infty)}$ for the extension, provided that the norm of the boundary trace is small:
\begin{proposition}[see Thm. \ref{smallenergyext}]
  There is a constant $\epsilon_0>0$ and a finite constant $C$ such that if
$$
\int_{\mathbb S^3}|\nabla\phi|^3\leq\epsilon_0, \phi:\mathbb S^3\to \mathbb S^3,
$$
then there exists $u\in W^{1,4}(B^4, \mathbb S^3)$ such that
$$
u=\phi\text{ on }\partial B^4\text{ in the sense of traces and }\|\nabla u \|_{L^4(B^4)}\leq C\|\nabla\phi\|_{L^3(\mathbb S^3)}.
 $$
\end{proposition}
This is part of our proof of Theorem B and is proved in Section \ref{sec:litenext} using a method in the spirit of \cite{Uhl2}, developed in Appendix \ref{sec:Uhlenbeck}.

\subsection{Existence of $W^{1,4}$-extension without norm bounds}\label{ssec:extw14nb}
As for the case of global gauges, we can in general obtain $W^{1,4}(B^4,\mathbb S^3)$-extensions once we give up the requirement to have a norm control of the extension like in Theorem B. This phenomenon represents one example of situations in which function spaces have a behavior which is more complex than what can be detected by only looking at their norms. 
\begin{proposition}\label{extensionnocontrol}
 If $\phi\in W^{1,3}(\mathbb S^3, \mathbb S^3)$ then its topological degree is well-defined, cfr. \cite{SchoenUhl} and \cite{whiteenmin}. Suppose then that $\op{deg}\phi=0$.\\

Then there exists $u\in W^{1,4}(B^4, \mathbb S^3)$ such that
$$
u=\phi\text{ on }\partial B^4\text{ in the sense of traces.}
$$
\end{proposition}
\begin{proof}
We use the extension as in the Section \ref{sec:modint}. The construction using Lemma \ref{courant}(Courant-Lebesgue analogue) is done on a series of domains $B(x_i,\rho_i)\cap B^4$ where $x_i\in\partial B^4, \rho_i\in[\rho_0,2\rho_0]$ for the choice 
$$
\rho_0:=\inf\left\{\rho>0\text{ s.t. }\exists x_0\in\partial B^4,\int_{B(x_0,2\rho)\cap\partial B^4}|\nabla\phi|^3\geq\epsilon_0\right\}.
$$
Note that we have no a priori control on how small $\rho_0$ could get, but it cannot be zero for a fixed $\phi$. Then a Lipschitz extension $u:\mathcal R\to \mathbb S^3$ to a Lipschitz region $\mathcal R$ included between $B^4\setminus B_{1-2\rho_0}$ and $B^4\setminus B_{1-\rho_0}$ would exist as in Section \ref{sec:modint} and such $u$ will also be Lipschitz (with constant blowing up at the rate $\sim\rho_0^{-1}$) and would have degree zero (the preservation of degree follows because the extension used in the construction preserves the homotopy type, cfr \cite{whiteenmin}). In particular we can do a further Lipschitz (thus $W^{1,4}$) extension to the interior of $B^4\setminus\mathcal R$. This provides the desired $u$.
\end{proof}

The proof of the above proposition is constructive, and no hint that the construction is optimal is available. In the next section we prove that actually \underbar{no} \underbar{general} \underbar{bound} \underbar{in} $W^{1,4}$ can be achieved, because of the intervention of the topological degree, much as in the case of $SU(2)$-instantons.

\subsection{Impossibility of $W^{1,4}$-bounds for an extension}\label{ssec:now14extglob}
\begin{proposition}\label{nocontrext}
There exists no finite function $f:\mathbb R^+\to\mathbb R^+$ such that for each $\phi\in W^{1,3}(\mathbb S^3, \mathbb S^3)$ there exists a function $u\in W^{1,4}(B^4, \mathbb S^3)$ satisfying
$$
u=\phi\text{ on }\partial B^4\text{ in the sense of traces and }\|\nabla u \|_{L^4(B^4)}\leq f\left(\|\nabla\phi\|_{L^3(\mathbb S^3)}\right).
$$
\end{proposition}

\begin{proof}
We recall the robustness if degree under strong convergence in $W^{1,3}(\mathbb S^3, \mathbb S^3)$ (see \cite{SchoenUhl, whiteenmin} and also \cite{brezisnir1,brezisnir2}). Consider $\phi=id_{\mathbb S^3}$, which has degree $1$. Suppose an extension $u:B^4\to \mathbb S^3$ to $\phi$ would exist with $\|u\|_{W^{1,4}}\leq C'$. It will be possible to approximate in $W^{1,4}$-norm $u$ by functions $u_i\in C^\infty(B^4, \mathbb S^3)$, since smooth functions are dense in $W^{1,4}(B^4,\mathbb S^3)$. In particular the degrees $\op{deg}(\phi_i)$ of $\phi_i=u_i|_{\partial B^4}$ will have to be zero. Thus it is not possible that $\phi_i\stackrel{W^{1,3}}{\to} \phi$ because the degree is preserved under strong $W^{1,3}$-convergence).

This proves the absence of a continuous extension operator. To show that also boundedness is impossible, we use a slightly different argument.\\

Consider $\phi_0\in W^{1,3}\cap C^\infty(\mathbb S^3,\mathbb S^3)$ which is a perturbation of the identity equal to the south pole $S$ in a neighborhood $N_S$ of $S$. Then consider a M\"obius transformation $F:\mathbb S^3\to \mathbb S^3$ such that $F^{-1}(N_S)$ includes the lower hemisphere, and consider $\phi'=\phi_0\circ F, \phi''=\phi_0\circ(-F)$. Then identifying $\mathbb S^3\sim SU(2)$ such that $S\sim id_{SU(2)}$ use the group operation to define $\phi=\phi'\phi''$. Note that $\|\phi \|_{W^{1,3}}\leq 2\|\phi_0\|_{W^{1,3}}$ since the conformal maps $F, -F$ preserve the energy; moreover $\phi$ has  zero degree.\\

Let $F_n$ be a family of M\"obius transformations symmetric about $S$ and such that they concentrate more and more near $S$ (with the notation of Appendix \ref{sec:moeblemmas} we may take $F_n:=F_{v_n}$ for $v_n=(1-1/n)S$). Define $\phi_n':=\phi'\circ F_n$ and $\phi_n=\phi_n'\phi''$. It is clear by conformal invariance of the $W^{1,3}$-energy that $\phi_n$ have constant energy. They converge weakly to $\phi''$ and have degree zero.\\

Call $u_n$ the extension of $\phi_n$ and suppose that $\|u_n\|_{W^{1,4}}\leq C$ independent of $n$. We may suppose that $u_n\stackrel{W^{1,4}}{\rightharpoonup} u_\infty\in W^{1,4}(B^4, \mathbb S^3)$ and we obtain $u_\infty|_{\partial B^4}=\phi''$ in the sense of traces. We then apply the result of \cite{whiteenmin} (see also \cite{SchoenUhl}) which in this case says that the $3$-dimensional homotopy class passes to the limit under bounded sequential weak $W^{1,4}(B^4,\mathbb S^3)$-limits. We obtain again a contradiction to boundedness since $\op{deg}(\phi'')=-1$ whereas the same degree is zero for the maps $\phi_n$.
\end{proof}

\subsection{Moving frames and their gauges}\label{ssec:helein}
We describe here a lifting problem arising in the theory of moving frames on $2$-dimensional surfaces, where the Lorentz spaces appear again in the optimal estimates. The model question is as follows:
\begin{question}
Suppose given a map (representing the normal vector of an immersed surface) $\vec{n}\in W^{1,2}(D^2, \mathbb S^2)$. Does there exist a $W^{1,2}$  \underbar{controlled} \underbar{trivialization} $\vec e=(\vec e_1,\vec e_2)$ of the pullback bundle $\vec n^{-1}T\mathbb S^2$? A trivialization is defined by two vector fields $\vec e_1,\vec e_2\in W^{1,2}(D^2,\mathbb S^2)$ such that the pointwise constraints $|\vec e_1|=|\vec e_2|=1, \vec e_1\cdot\vec e_2=0$ are satisfied almost everywhere and $\vec n=\vec e_1\times \vec e_2$.
\end{question}
This problem behaves like the one of global controlled gauges, namely for small energy a lift exists and is controlled, and for large energy lifts can be found but with no general control. Ulenbeck's $\epsilon$-regularity estimate is mirrored in the following Theorem. This result, was proved initially by F. H\'elein under the hypothesis $\|\nabla\vec n\|_{L^2}\leq C$ and improved by Y. Bernard and T.Rivi\`ere who proved that it is enough to assume a smallness condition in weak-$L^2$:
\begin{theorem}[\cite{bernardriv} Lemma IV.3, cfr. also \cite{helein} Lemma 5.1.4]
There exists $\epsilon_0$ such that if $\|\nabla\vec n\|_{L^{2,\infty}}\leq\epsilon_0$ then there exists a trivialization, with the control 
\[
\|\nabla \vec e_1\|_{L^2} +\|\nabla \vec e_2\|_{L^2}\leq C\|\nabla \vec n\|_{L^2}.
\]
and
\[
 \|\nabla \vec e_1\|_{L^{2,\infty}} +\|\nabla \vec e_2\|_{L^{2,\infty}}\leq C\|\nabla \vec n\|_{L^{2,\infty}}.
\]
\end{theorem}
Note that for the improvement above, the $L^2$-energy might blow up, yet still control the energy of the trivialization, as long as we stay small in Lorentz norm. It would be interesting to explore this kind of phenomenon also for curvatures in higher dimensions like in our setting.\\

The bad behavior in case of large energy regime starts at the energy level $8\pi$ (and this is optimal, see \cite{kuwertli}). This number has an evident topological significance, because if $\vec n$ is homotopically nontrivial, i.e. parameterizes a non-contractible $2$-cell of $\mathbb S^2$ then $4\pi=|\mathbb S^2|\leq\int_{D^2} u^*(d \op{Vol}_{\mathbb S^2})\leq\frac{1}{2}\int_{D^2}|\nabla\vec n|^2$, so $8\pi$ is the smallest energy of a topologically nontrivial $\vec n$.\\
We also have the following lemma, similar to Section \ref{ssec:now14extglob}:
\begin{lemma}
 For $\int |\nabla\vec n|^2>8\pi$ there can be no controlled $W^{1,2}$ trivialization $\vec e$.
\end{lemma}
\begin{proof}[Sketch of proof:]
 We choose $\vec n$ mapping a neighborhood $D^2\setminus B_r:=N_1$ for small $r$ to the south pole of $\mathbb S^2$, has degree $1$ and equals a conformal map outside a small neighborhood $N_2\Supset N_1$. Such $\vec n$ exists with energy as close as wanted to $8\pi$, independently of $r$ by conformal invariance of the energy.\\
Supposing a trivialization $\vec e=(\vec e_1,\vec e_2)$ exists, on $N_1$ it will span the ``horizontal'' $2$-plane of $\mathbb R^3$ which is perpendicular to $S=(0,0,-1)$. On circles $\partial B_\rho,\rho>r$ by Fubini theorem for almost all $\epsilon$ we will have that $\vec e_i,i=1,2$ will be $W^{1,2}$ thus $C^0$ and they have values in the equator of $\mathbb S^2$. By well-posedness of the topological degree and since $\vec n$ is nontrivial in homotopy, we obtain that each $e_i$ will make a full turn on each $\partial B_r$. This gives that $\int_{\partial B_r}|\nabla \vec e_i|\geq 1$ on $\partial B_r$ and by Jensen's inequality we obtain 
$$
\int_{D^2\setminus B_r}|\nabla\vec e_i|^2\geq C\int_r^1\frac{1}{\rho^2}\rho d\rho\geq C\left|\log \frac{1}{r}\right|
$$
since there is no positive lower bound of $r>0$, we see that we cannot have a controlled trivialization.
\end{proof}

There is an analogue also of our $W^{1,(4,\infty)}$ extension result here, and it corresponds to taking the so-called ``Coulomb frames''. The result is a general estimate with no restriction on $\vec n$, but with the Lorentz norm $L^{(2,\infty)}$ instead of the $L^2$ norm (this estimate follows from Wente's \cite{wente} inequality using \cite{adams}):
\begin{proposition}[\cite{rivcourse}, VII.6.3]
 Let $\vec n\in W^{1,2}(D^2, \mathbb S^2)$. Then there exist a trivialization $\vec e$ belonging to $W^{1,(2,\infty)}$ exists, which satisfies the Coulomb condition 
$$
\op{div}\langle \vec e_1,\nabla\vec e_2\rangle=0
$$
and the control
$$
\|\nabla \vec e_1\|_{L^{(2,\infty)}}+\|\nabla \vec e_2\|_{L^{(2,\infty)}}\lesssim \|\nabla\vec n\|_{L^2} + \|\nabla\vec n\|_{L^2}^2.
$$
\end{proposition}

\section{The Hopf lift extension}\label{sec:hopflift}
\noindent
We prove here the Theorem C. We consider a fixed $\phi\in W^{1,2}(\mathbb S^2, \mathbb S^2)$ and we need to construct an extension $u\in W^{1,(3,\infty)}(B^3, \mathbb S^2)$ such that
\begin{equation*}
 \|u\|_{W^{1,(3,\infty)}(B^3)}\lesssim \|\phi\|_{W^{1,2}(\mathbb S^2)}(1+\|\phi\|_{W^{1,2}(\mathbb S^2)}),
\end{equation*}
where the implicit constant is independent of $\phi.$\\

The strategy of proof uses a construction based on the Hopf fibration which has been introduced in \cite{hr1}. The same strategy has been later on performed in \cite{bc} for proving similar lifting results as in \cite{hr1}. In the smooth case we will first lift $\phi:\mathbb S^2\to \mathbb S^2$ to $\tilde \phi:\mathbb S^2\to \mathbb S^3$ such that $H\circ \tilde\phi=\phi$ where $H:\mathbb S^2\to \mathbb S^3$ is the Hopf fibration. Then we will extend $\tilde \phi$ by using a Lorentz analogue of \ref{trickhkl}(projection trick), working with similar conditions on dimensions and exponents. projecting back to $\mathbb S^2$ via $H$ will keep the estimates.\\

Before the proof, we recall some properties of the map $H$.

\subsection{Facts about the Hopf fibration}
Identifying $\mathbb S^3$ with the unit sphere of $\mathbb C^2$, with complex coordinates $(Z,W)$, the Hopf projection is $H(Z,W)=Z/\bar W$ and its fibers are maximal circles. This gives a function with values in $\mathbb C\cup \{\infty\}\simeq \mathbb S^2$. If we look at $\mathbb S^3\subset \mathbb R^4$ with the inherited coordinates $(x_1,x_2,x_3,x_4)$ then we can identify 
\begin{equation}\label{exactnesspullback}
 H^*\omega_{\mathbb S^2}=d\alpha,\quad \text{for }\alpha=\frac{1}{2}(x_1 dx_2 - x_2 dx_1 +x_3dx_4 - x_4 dx_3). 
\end{equation}
Here $\omega_{\mathbb S^2}$ is a constant multiple of the volume form of $\mathbb S^2$. Since $\mathbb S^1\sim U(1)$ we can regard $\mathbb S^3\stackrel{H}{\to}\mathbb S^2$ as a principal $U(1)$-bundle $P\to\mathbb S^2$.\\

Let $\phi:\mathbb C\to \mathbb S^2$ be a smooth function. Then $d(\phi^*\omega_{\mathbb S^2})=0$ because $\Omega^3(\mathbb R^2\simeq\mathbb C)=\{0\}$. Since $H^2_{dR}(\mathbb C)=0$ there exists a $1$-form $\eta$ such that
\begin{equation}\label{existsgauge}
 d\eta=\phi^*\omega_{\mathbb S^2}.
\end{equation}
We also note that for a smooth $\phi:\mathbb C\to \mathbb S^2$ the pullback of the $U(1)$-bundle $P$ is trivial, since $\mathbb R^2$ is contractible. A trivialization of the bundle $\phi^*P\to\mathbb C$ can be identified with a lift $\tilde \phi$ of $\phi$. From the equation \eqref{exactnesspullback} we can deduce that $d\eta=\tilde \phi^*H^*\omega_{\mathbb S^2} =\tilde \phi^*d\alpha=d( \tilde\phi^*\alpha)$ and again there exists a $1$-form $\tilde\eta$ as in \eqref{existsgauge}, defined by 
\begin{equation}\label{gaugefromlift}
 \tilde \eta=\tilde \phi^*\alpha.
\end{equation}
$\tilde\eta$ coincides with $\eta$ up to adding an exact form $d\theta$: we have $\tilde\phi^*\alpha -\eta=d\phi$. If we come back to the bundle point of view then $d\theta$ represents the effect of change of coordinates of the trivialization giving $\tilde\phi$, i.e. of a change of gauge. We have then $\eta=\tilde\phi^*\alpha - d\theta=(e^{-i\theta}\tilde \phi)^*\alpha$, where the action of $e^{-i\theta}$ is intended as an $U(1)$-gauge change and $\theta:\mathbb C\to \mathbb R$ is determined up to a constant. Moreover, since $DH$ is an isometry between the orthogonal complement of the tangent space of the fiber $T_pH^{-1}(H(p))$ and $T_p\mathbb S^2$, we also obtain the following norm identity:
\begin{equation}\label{normidentity}
 |D\tilde \phi|^2=|\tilde \eta|^2 +|D\phi|^2.
\end{equation}
\subsection{Hopf lift with estimates}
We start the proof of Theorem C with the following first step:
\begin{proposition}\label{exthopf}
 Suppose $\phi\in W^{1,2}(\mathbb C, \mathbb S^2)$. Then there exists a lifting $\tilde \phi:\mathbb C\to \mathbb S^3$ such that $H\circ\tilde \phi=\phi$ and there exists a universal constant $C$ such that
$$
||\nabla \tilde \phi||_{L^{2,\infty}}\leq C||\nabla \phi||_{L^2}(1+||\nabla \phi||_{L^2}).
$$
\end{proposition}
\begin{proof}[Proof of Proposition \ref{exthopf}:]
The proof is divided in two steps.\\
\textbf{Step 1. }\emph{Constructions in the smooth case.} 
We have seen that, at least in the smooth case, constructing a $1$-form $\eta$ as in \eqref{existsgauge} is equivalent to the construction of a lift $\tilde \phi:\mathbb C\to \mathbb S^3$. We now observe that such a $1$-form can be in turn easily constructed, by inverting the Laplacian on $\mathbb C$, via its Green kernel, which is of the form $K(x)=-\gamma\log|x|$. In particular $K\in W^{1,(2,\infty)}$, which is the reason why this norm appears). First note that $dd^*(K*\beta)=0$ for a smooth $L^1$-integrable $2$-form $\beta$ on $\mathbb C$. We can then use this formula for $\beta=\phi^*\omega_{\mathbb S^2}$, and taking into account the fact that $\nabla K$ is in $L^{2,\infty}$, by the Lorentz space Young inequality (see \cite{grafakos}) we obtain that the $1$-form $\eta$ defined as
\begin{equation}\label{definitionofeta}
 \eta:=d^*\left[K*(\phi^*\omega_{\mathbb S^2})\right],\quad \eta\to 0\text{ at infinity}
\end{equation}
satisfies \eqref{existsgauge} and the estimates
\begin{equation}\label{lorentzestimates}
 ||\eta||_{L^{2,\infty}}\lesssim||\phi^*\omega_{\mathbb S^2}||_{L^1}\lesssim||D\phi||_{L^2}^2||\phi||_{L^\infty}\simeq||D\phi||_{L^2}^2.
\end{equation}
We have mentioned where to find the proof that $\eta$ corresponds up to a unitary transformation to a lift $\tilde \phi$, and from \eqref{normidentity} and from \eqref{lorentzestimates} we also obtain the estimate for $\tilde \phi$ which reads as follows:
\begin{equation}\label{estimatetildeu}
 ||D\tilde \phi||_{L^{2,\infty}}\lesssim||\eta||_{L^{2,\infty}}+||D\phi||_{L^2}\lesssim ||D\phi||_{L^2}(1+||D\phi||_{L^2}).
\end{equation}

\textbf{Step 2. }\emph{Extending the constructions to $W^{1,2}$.} 
The results obtained so far apply for $\phi\in C^\infty(\mathbb C,\mathbb S^2)$. We use the by now well-known fact that while not dense in the strong topology, the functions in $C^\infty(\mathbb C,\mathbb S^2)$ are instead \emph{dense with respect to the weak sequential convergence} (see \cite{Bethuel, hl2}). The constraint of $u_n$ having values in $\mathbb S^2$, as well as the constraint $\tilde \phi_n\circ H=\phi_n$ for the $\tilde \phi_n$, are pointwise constraints (note indeed that the function $H$ is smooth), so they are preserved under weak convergence $\phi_n\rightharpoonup \phi\in W^{1,2}$. Now we state the only less classical point in the following lemma.
\begin{lemma}\label{weaklimweakl2}
 $L^{2,\infty}$-estimates are preserved under weak convergence in $L^2$. In other words, if $f_n\in L^2$ are weakly convergent to $f\in L^2$ then $||f||_{L^{2,\infty}}\leq\liminf_{n\to\infty}||f_n||_{L^{2,\infty}}$.
\end{lemma}
\begin{proof}[Proof of the lemma:]
We observe that a positive answer to this question cannot directly and trivially be obtained by interpolation, since $L^\infty$-norm is not lower semicontinuous with respect to weak convergence in $L^2$. We thus proceed by duality, namely we note that 
$$
L^{(2,\infty)}=\left(L^{(2,1)}\right)'\text{ and }L^{(2,1)}\subset L^2.
$$
Therefore $\langle f_n,\phi\rangle\to\langle f,\phi\rangle$ for all $\phi\in L^{(2,1)}$ and by usual Banach space theory we obtain the thesis.
\end{proof}
Applying the Lemma, we obtain the wanted estimate via Bethuel's weak density result.
\end{proof}

We observe that given a map $\phi\in W^{1,2}(\mathbb S^2,\mathbb S^2)$, we can obtain a map $u:\mathbb C\to \mathbb S^2$ having the same norm by composing with the inverse stereographic projection $\Psi^{-1}:\mathbb C\to \mathbb S^2$: we use here the facts that the exponent $2$ is equal to the dimension, and that $\Psi$ is conformal. In a similar way, having constructed a lift $\tilde u:\mathbb C\to \mathbb S^3$, we obtain automatically a lift $\tilde \phi$ of $\phi$ by composing back with $S$. The same reasoning using conformality also shows that the $L^{2,\infty}$-norm of the gradient of $\tilde \phi$ is preserved. This proves the following:
\begin{corollary}\label{corollhopf}
  Suppose $\phi\in W^{1,2}(\mathbb S^2, \mathbb S^2)$. Then there exists a lifting $\tilde \phi:\mathbb S^2\to \mathbb S^3$ such that $H\circ\tilde \phi=\phi$ and there exists a universal constant $C$ such that
$$
||\nabla \tilde \phi||_{L^{2,\infty}}\leq C||\nabla \phi||_{L^2}(1+||\nabla \phi||_{L^2}).
$$
\end{corollary}

\subsection{Projection and wise choice of the point}
To proceed in our strategy for the proof of Theorem C, we use a version of the projection trick of Section \ref{ssec:hardtlintrick}.

\begin{proposition}[projection trick 2]\label{hltrick}
Suppose that $\tilde \phi\in W^{1,(2,\infty)}(\mathbb S^2, \mathbb S^3)$. Then there exists a function $\tilde u:B^3\to \mathbb S^3$, such that $\tilde u|_{\partial B^3\setminus \mathbb S^2}=\tilde \phi$ and satisfying the following bounds for some universal constant $C$
$$
||\tilde u||_{W^{1,(3,\infty)}(B^3)}\leq C||\tilde \phi||_{W^{1,(2,\infty)}(\mathbb S^2)}.
$$
\end{proposition}
\begin{proof} We proceed in two steps, of which the first one introduces the $W^{1,(3,\infty)}$-norm estimate, and the second one ensures that the constraint of having values in $\mathbb S^3$ can be preserved.\\

\textbf{Step 1.}\emph{Harmonic extension.} 
Consider a solution $\tilde u$ of the following equation:
\begin{equation}\label{HE1}
 \left\{\begin{array}{ll}
       \Delta \tilde u = 0\text{ on }B^3,\\
       \tilde u = \tilde \phi\text{ on }\partial B^3.  
        \end{array}
\right.
\end{equation}
By using the Poisson kernel estimates we obtain that $\tilde u\in W^{1,(3,\infty)}(B^3, B^4)$ and 
\begin{equation}\label{HE2}
\|\nabla\tilde u\|_{L^{(3,\infty)}}\lesssim \|\nabla \tilde\phi\|_{L^{(2,\infty)}}.
\end{equation}
\textbf{Step 2. }\emph{Projection in the target.} 
We now correct the fact that $\tilde u$ has values not in $\mathbb S^3$ but in its convex hull $B^4$. For $a\in B_{1/2}^4$ we note $\pi_a$ the radial projection $\pi_a:B^4\to \mathbb S^3$ of center $a$, i.e. 
$$
\pi_a(x):=a+t_{a,x}(x-a),\text{ for }t_{a,x}\geq 0\text{ such that }|\pi_a(x)|=1.
$$
In order to estimate the norm of $u_a:=\pi_a\circ \tilde u$ we note that 
$$
|\nabla(\pi_a\circ \tilde u)|(x)\lesssim \frac{|\nabla \tilde u(x)|}{|u(x)-a|},\
$$
with an implicit constant bounded by $4$ as long as $a\in B^4_{1/2}$. We just estimate the $L^p$-norm of $\nabla u_a$ for $p\in[1,4[$. We note that $\int_{B_{1/2}}|\tilde u (x)-a|^{-p}da$ is bounded for all such $p$ by a number $C_p$ independent of $x$, therefore by changing the order of integration and applying Fubini, we obtain
$$
\int_{B_{1/2}}\int_{B_1}|\nabla u_a(x)|^p dxda\leq C_p\int_{B_1}|\nabla \tilde u(x)|^p\int_{B_{1/2}}|\tilde u(x)-a|^{-p}da\leq C_p||\nabla \tilde u||_p^p.
$$
In other words, the assignment $a\mapsto u_a$ gives a map whose $L^1_a(B_{1/2},W^{1,p}_x(B^3,\mathbb S^3))$-norm is bounded by the $L^p$-norm of $\nabla\tilde u$ for $p\in[1,4[$. First observe that by Lions-Peetre reiteration $L^{(3,\infty)}$ is an interpolation between $L^{p_0}$ and $L^{p_1}$ with $3\in ]p_0,p_1[\subset]1,4[$. We now use the nonlinear interpolation theorem of Tartar. Call $U(a,x):=\frac{\nabla \tilde u(x)}{|\tilde u(x)-a|}$. We know that the map $u\mapsto U$ is bounded between $W^{1,p_i}$ and $L^{p_i}$ for $i=0,1$. In order to show that it also satisfies 
\begin{equation}\label{linfest}
\sup_{\lambda>0}\lambda^3\left|\left\{(x,a)\in B_1\times B_{1/2}:\:\frac{|\nabla u(x)|}{|u(x)-a|}>\lambda\right\}\right|=\|U\|_{L^{(3,\infty)}}^3\lesssim\|\tilde u\|_{W^{1,(3,\infty)}}^3
\end{equation}
we will check the local estimate 
$$
\left\|\frac{\nabla u(x)}{|u(x)-a|}-\frac{\nabla v(x)}{|v(x)-a|}\right\|_{L^{p_1}}\lesssim\|u-v\|_{L^{p_1}}.
$$
This follows since 
\[
\begin{split}
\int_{B_1}\int_{B_{1/2}}&\left|\frac{\nabla u(x)}{|u(x)-a|}-\frac{\nabla v(x)}{|v(x)-a|}\right|^{p_1}\\
&\lesssim\int_{B_1}|\nabla u -\nabla v|^{p_1}\int_{B_{1/2}}\left(|u(x)-a|^{-p_1} + |v(x)-a|^{-p_1}\right)da\: dx
\end{split}
\]
and to the second factor the same estimates as before apply, uniformly in $x$. Thus \eqref{linfest} holds. From \eqref{linfest} it easily follows that there exists $a\in B_{1/2}$ for which 
\begin{equation}\label{estproj}
\|\nabla u_a\|_{L^{(3,\infty)}(B_1)}\lesssim\|\tilde u\|_{W^{1,(3,\infty)}}.
\end{equation}
Combining \eqref{HE2} and \eqref{estproj}, we obtain the claim of the proposition, for $\hat u:=u_a$.
\end{proof}
\subsection{End of proof}
\begin{proof}[Proof of Theorem C:]
Apply consecutively Corollary \ref{corollhopf} and Prop. \ref{hltrick}(projection trick 2). For this $\hat u$ as in Prop. \ref{hltrick} we can then consider $u:=H\circ u_a:B^3\to \mathbb S^2$. Since $H$ is Lipschitz we obtain the pointwise estimate
\begin{equation}\label{Hlip}
 |\nabla u|\lesssim |\nabla u_a|.
\end{equation}
Combining this with the estimates of Corollary \ref{corollhopf} and Prop. \ref{hltrick}(projection trick 2) we obtain the thesis of Theorem C.
\end{proof}
\subsection{Modification of proof in the case of $W^{1,p}(\mathbb S^m,\mathbb S^2)$}
In this section we prove Theorem D and Proposition \ref{sms2}. 
\begin{proof}[Proof of Theorem D and of proposition \ref{sms2}]
We consider here $n=2<m$ and $\frac{3m}{m+1}\leq p<\frac{4m}{m+1}$ as in Proposition \ref{sms2}. We will use the fact that such $p$ is always $>2$. The construction of the $1$-form $\eta$ satisfying \eqref{gaugefromlift} and \eqref{normidentity} can be done in a completely analogous way if the domain is $\mathbb R^m,m\geq 3$. The only difference is that in such case the Laplacian on $2$-forms like $\phi^*\omega_{\mathbb S^2}$ has the form $\delta=d^*d+dd^*$ where the first part does not vanish anymore. In this case however we may still solve
\[
 \left\{\begin{array}{l}
         d\eta=\phi^*\omega_{\mathbb S^2},\\
         d^*\eta=0,\\
         \eta(x)\to 0,\quad |x|\to\infty.
        \end{array}
\right.
\]
If $\phi\in W^{1,p}(\mathbb R^m,\mathbb S^2)$ and since $p>2$ we then have 
\[
  \|d\eta\|_{L^{p/2}(\mathbb R^m)}\leq C\|\phi^*\omega_{\mathbb S^2}\|_{L^{p/2}(\mathbb R^m)}\leq C\| d\phi\|_{L^p(\mathbb R^m)}^2.
\]
As before we have \eqref{normidentity}, from which we also obtain $ |D\tilde\phi|^p\lesssim|\eta|^p+|D\phi|^p$. Passing to $\mathbb S^m$ and noting that in dimension $m\geq p$ there holds $W^{1,p/2}(\mathbb  S^m,\mathbb S^2)\hookrightarrow L^{\frac{mp}{2m-p}}(\mathbb  S^m,\mathbb S^2)\hookrightarrow L^p(\mathbb  S^m,\mathbb S^2)$ we obtain
\[
\|D\tilde\phi\|_{L^p(\mathbb S^m,\mathbb S^2)}\lesssim\|D\phi\|_{L^p(\mathbb S^m,\mathbb S^2)}^2+\|D\phi\|_{L^p(\mathbb S^m,\mathbb S^2)}.
\]
Harmonic extension and Prop. \ref{trickhkl}(projection trick) allow then to obtain an extension $\tilde u:B^{m+1}\to\mathbb S^2$ of $\tilde\phi$ such that
\[
\|\nabla\tilde u\|_{L^{\frac{m+1}{m}p}(B^{m+1},\mathbb S^3)}\lesssim\|D\tilde\phi\|_{L^p(\mathbb S^m,\mathbb S^3)},
\]
provided $\frac{m+1}{m}p<4$ (which is the condition appearing in Prop. \ref{trickhkl}(projection trick). Composing with the Hopf map $H$ at most decreases the norm, thus we obtain that $u:=H\circ\tilde u$ is the wanted controlled extension as in Proposition \ref{sms2} and in Theorem D (note that for $m=3$ the condition $\frac{m+1}{m}p<4$ is equivalent to $p<3$).
\end{proof}
\section{The extension theorem for $W^{1,3}$ maps $\mathbb S^3\to \mathbb S^3$}\label{sec:extw14}
\noindent
This section is devoted to the proof of the following theorem:
\begin{theoremB''}\label{extension}
There exists a constant $C>0$ with the following property. Suppose $\phi\in W^{1,3}(\mathbb S^3,\mathbb S^3)$. then there exists an extension $u\in W^{1,(4,\infty)}(B^4,\mathbb S^3)$ of $\phi$ such that the following estimate holds:
\begin{equation}\label{eq:estimate}
 \|\nabla u\|_{L^{4,\infty}(B^4)}\leq C\left(e^{C\|\nabla\phi\|_{L^3}^9} + e^{C\|\nabla\phi\|_{L^3}^6}\|\nabla\phi\|_{L^3}\right).
\end{equation}
\end{theoremB''}
\subsection{Modulus of integrability estimates}\label{sec:modint}
In general during our estimates we indicate by $C$ a positive constant, which may change from line to line, and also within the same line. We start by fixing the notation for the main quantity which will be used control the energy concentration of our maps.
\begin{definition}
If $D\subset\mathbb R^4$ and $f:D\to\mathbb R$ is measurable then let $E(f,\rho,D)$ denote the (possibly infinite) modulus of integrability of $f$, which is defined as
$$
E(f,\rho,D)=\sup_{x\in D}\int_{B_\rho(x)\cap D}|f|.
$$
\end{definition}
The modulus of integrability fits into a sort of elliptic estimate as follows.
\begin{proposition}[integrability modulus estimates]\label{equiintballs}
 Let $\phi\in W^{1,3}(\partial B^4, \mathbb S^3)$ and assume that $u$ is the solution to the following equation:
$$
\left\{\begin{array}{ll}
        \Delta u=0&\text{ on }B^4,\\
         u=\phi&\text{ on }\partial B^4.
       \end{array}
\right.
$$
Then there exists a constant $C_1$ independent of $\phi,\rho$ such that when $\rho\in]0,1/4[$ the following inequality holds true:
\begin{equation}\label{equiintest}
 E(|\nabla u|^4,\rho,B^4)\leq C_1 E(|\nabla \phi|^3,2\rho,\partial B^4)^{1/3}\int_{\partial B^4}|\nabla \phi|^3.
\end{equation}
\end{proposition}
\begin{proof}
We have to prove that for all $x_0\in B^4$,
\begin{equation}\label{boundarycenter}
\int_{B_\rho(x_0)\cap B^4}|\nabla u|^4\leq C_1 E(|\nabla \phi|^3,2\rho,\partial B^4)\int_{\partial B^4}|\nabla \phi|^3.
\end{equation}
\textbf{Step 1.} We prove \eqref{boundarycenter} for $x_0\in \partial B^4$.
$$
\int_{B_\rho(x_0)\cap B^4}|\nabla u|^4\leq C_0 E(|\nabla \phi|^3,2\rho,\partial B^4)\int_{\partial B^4}|\nabla \phi|^3.
$$
The function $u$ can be obtained by superposition, using a cutoff function $\eta:\mathbb S^3\to[0,1]$ which equals $1$ on $B_\rho(x_0)\cap \mathbb S^3$ and $0$ outside $B_{2\rho}(x_0)$ and satisfies $|\nabla\eta|\lesssim \rho^{-1}$. We will use the functions
$$
\left\{\begin{array}{ll}\Delta u_1=0&\text{ on }B^4,\\ u_1=\eta\phi:=\phi_1&\text{ on }\partial B^4.\end{array}\right. \left\{\begin{array}{ll}\Delta u_2=0&\text{ on }B^4,\\ u_2=(1-\eta)\phi:=\phi_2&\text{ on }\partial B^4.\end{array}\right.
$$
We can estimate these two functions separately because there holds 
$$
\int_{B_\rho(x_0)\cap B^4}|\nabla u|^4\lesssim \int_{B_\rho(x_0)\cap B^4}|\nabla u_1|^4 + \int_{B_\rho(x_0)\cap B^4}|\nabla u_2|^4.
$$
It is convenient to estimate separately the contributions of $u_1$ on $S'=B_{2\rho}(x_0)\cap \mathbb S^3$ and of $u_2$ on $S''=\mathbb S^3\setminus B_\rho(x_0)$; on $S''$ we use the Poisson formula and on $S'$ we use elliptic estimates.\\
By elliptic theory and the definition of $\eta$,
$$
\int_{B_\rho(x_0)\cap B^4}|\nabla u_1|^4\lesssim \left(\int_{S'}|\nabla\phi|^3\right)^{4/3}.
$$
Poisson's formula gives
$$
u_2(x)=C(1-|x|^2)\int_{\partial B^4}\frac{\phi_2(y)}{|x-y|^4}dy,
$$
therefore (using also the bound on $\eta$) we obtain a pointwise bound, in case $x\in B_\rho(x_0)\cap B^4, \rho<1/4$:
$$
|\nabla u_2|(x)\lesssim \rho\int_{S''}\frac{|\nabla\phi|}{|x-y|^4}dy + \int_{S''}\frac{|\phi|}{|x-y|^4}dy\lesssim\rho\int_{S''}\frac{|\nabla\phi|}{|x-y|^4}dy.
$$
Patching together the estimates obtained so far, we write
\begin{equation}\label{first}
 \int_{B_\rho(x_0)\cap B^4}|\nabla u|^4\lesssim \left(\int_{S'}|\nabla \phi|^3\right)^{4/3} + \rho^8\left(\int_{S''}\frac{|\nabla\phi|}{|x-y|^4}\right)^4=I+II,
\end{equation}
where the factor $\rho^8$ comes from the pointwise estimate for $\nabla u_2$ keeping in mind that $|B_\rho(x_0)\cap B^4|\lesssim \rho^4$.\\
The first summand is estimated as needed:
$$
I\leq\left(\int_{B_{2\rho}(x_0)\cap\partial B^4}|\nabla\phi|^3\right)^{1/3}\int_{\mathbb S^3}|\nabla\phi|^3\leq E(|\nabla \phi|^3,2\rho,\partial B^4)\int_{\mathbb S^3}|\nabla\phi|^3.
$$
To estimate $II$ we consider a cover of $S''$ by (finitely many) balls $B_\rho^i=B_{2\rho}(x_i)$ such that $x_i$ form a maximal $2\rho$-separating net and they are at distance at least $\rho$ from $x_0$. We use the estimate
$$
\int_{B_{2\rho}^i}|\nabla\phi|\leq |B_{2\rho}^i|\left(\fint_{B_{2\rho}^i}|\nabla\phi|^3\right)^{1/3},
$$
and the fact that for $y\in B_{2\rho}^i$ and $x\in B_\rho(x_0)\cap B^4$ there holds $|x-y|\gtrsim \op{dist}(x_i,x_0)$. The second summand of \eqref{first} can then be estimated as follows:
$$
II\lesssim \rho^8\left(\sum_i \op{dist}^{-4}(x_i, x_0)\rho^3 a_i^{1/3}\right)^4
$$
where $a_i=\fint_{B_{2\rho}^i}|\nabla\phi|^3$. We can use the expression $1/3=1/4+1/12$ for the exponent of $a_i$ together with a H\"older inequality to obtain:
\begin{eqnarray*}
II&\lesssim&\rho^8\sup_ia_i^{1/3}\left(\sum_i\op{dist}^{-4}(x_i,x_0)\rho^3a_i^{1/4}\right)^4\\
&\lesssim&\rho^{20}\left(\sup_i a_i^{1/3}\right)\left(\sum_i a_i\right)\left(\sum_i\op{dist}^{-\frac{16}{3}}(x_i,x_0)\right)^3.
\end{eqnarray*}
Now the first parenthesis is estimated by $\rho^{-1}E(|\nabla \phi|^3,2\rho,\partial B^4)$, the second one by $ \rho^{-3}\int_{\mathbb S^3}|\nabla\phi|^3$, and for the last factor we have the elementary estimate 
$$
\sum_i\op{dist}^{-\frac{16}{3}}(x_i,x_0)\lesssim\frac{1}{\rho^3}\int_{\mathbb S^3}\frac{dx}{|x-x_0|^{16/3} +\rho^{16/3}}\lesssim \rho^{-\frac{16}{3}}.
$$
These new estimates give
$$
II\lesssim\rho^{20} \rho^{-1}E(|\nabla \phi|^3,2\rho,\partial B^4)\rho^{-16}\rho^{-3}\int_{\mathbb S^3}|\nabla\phi|^3\lesssim E(|\nabla \phi|^3,2\rho,\partial B^4)\int_{\mathbb S^3}|\nabla\phi|^3.
$$
This gives the wanted estimate for $II$, finishing the proof of \eqref{boundarycenter} in the case $x_0\in\partial B^4$. Note that the constants introduced in our inequalities can be chosen independent of $\rho$ and are independent of $\phi$. Thus $C_0$ is also independent of these data.\\

\textbf{Step 2.} We now observe that we can reduce the case of $|x_0|<1$ to the treatment of Step 1, up to changing the constant $C_0$ in our estimate from Step 1.\\

If $|x_0|<1-2\rho$ then we can directly apply the estimates for the term $II$ of \eqref{first}, since now the denominator $|x-y|$ in the Poisson formula will be at least $\rho$ for all $x\in B_\rho(x_0)$. \\

The estimate of Step 1 also holds for $\rho>1/4$ with the same constant. We can cover the case $|x_0|\in]1-2\rho,1[$ with $\rho<1/4$ by noticing that if $x_0'=x_0/|x_0|$ then $B_{3\rho}(x_0')\supset B_\rho(x_0)$ and that the measures $|\nabla \phi|^3d\sigma,|\nabla u|^4dx$ are doubling with constants bounded by the packing constants of $\mathbb S^3$ and of $B^4$ respectively, while the function $E(f,\rho, D)$ is increasing in $\rho$. Therefore the inequality \eqref{boundarycenter} also holds for this last choice of $x_0$ up to changing $C_0$ by a factor depending only of the above packing constants.
\end{proof}

\subsection{Extension in the case of small energy concentration}\label{sec:litenext}

The following two lemmas will be used for the harmonic extension of a boundary value $\phi\in W^{1,3}(\mathbb S^3,\mathbb S^3)$ under the small concentration hypothesis of Proposition \ref{equiintballs}:
\begin{lemma}\label{l:fubini}
 If $u\in W^{1,4}(B^4,\mathbb R^4)$ and $\rho\in ]0,1/2[, x_0\in\partial B^4$ then there exists $\bar\rho\in[\rho,2\rho]$ such that 
$$
\bar\rho\int_{\op{int}(B^4)\cap\partial B_{\bar\rho}(x_0)}|\nabla u|^4\leq C\int_{B^4 \cap B_\rho(x_0)}|\nabla u|^4.
$$
\end{lemma}
\begin{proof}
We just use the mean value theorem together with the following computation:
$$
\int_\rho^{2\rho}\int_{\op{int}(B^4)\cap\partial B_{\rho'}(x_0)}|\nabla u|^4 d\rho'=\int_{B_{2\rho}\setminus B_\rho(x_0)}|\nabla u|^4\leq\int_{B^4 \cap B_\rho(x_0)}|\nabla u|^4.
$$
\end{proof}
\begin{lemma}[Courant-Lebesgue analogue]\label{courant}
Fix $\bar\rho\in]0,1[$. There exists a constant $C>0$ such that if $u\in W^{1,4}(B^4, \mathbb R^4)$ is the extension of $\phi\in W^{1,3}(\mathbb S^3,\mathbb S^3)$ and if
$$
\bar\rho\int_{\op{int}(B^4)\cap\partial B_{\bar\rho}(x_0)}|\nabla u|^4\leq C
$$
with $x_0\in\partial B^4$, then for almost every $x\in \partial\left(B^4\cap B_{\bar\rho}(x_0) \right)$ there holds
\begin{equation}\label{eq:distbdry}
 \op{dist}(u(x),\mathbb S^3)\leq\frac{1}{8}.
\end{equation}
\end{lemma}
\begin{proof}
 Note that the hypotheses $x_0\in\partial B^4,\bar\rho<1$ have the following two geometric consequences: (1) $\partial B^4\cap\partial B_{\bar\rho}(x_0)$ has positive measure; (2) $B^4\cap B_{\bar\rho}(x_0)$ is $2$-bilipschitz equivalent to $B_{\bar\rho}$. Therefore we may just prove that \eqref{eq:distbdry} holds true on $\partial B_{\rho}$ for a function such that
$$
\left\{\begin{array}{l}
        \bar\rho\int_{\partial B_{\bar\rho}}|\nabla u|^4<C,\\
        \left|\left\{x:\:|u|(x)=1\right\}\right|>0.
       \end{array}
\right.
$$
To do this note that by definition $u(x)\in \mathbb S^3$ for a.e. $x\in \partial B^4$, then use the Sobolev inequality 
$$
\|u\|_{C^{0,1/4}(\partial B_{\bar\rho})}^4\lesssim \bar\rho\int_{\partial B_{\bar\rho}}|\nabla u|^4,
$$
valid in dimension $3$. For $C$ small enough we obtain \eqref{eq:distbdry}.
 \end{proof}

The next theorem is inspired by Uhlenbeck's technique for the removal of singularities of Yang-Mills fields. We postpone its proof to Appendix \ref{sec:Uhlenbeck}. See Theorem \ref{smallenergyext2}(small energy extension) for an equivalent statement.
\begin{theorem}[Uhlenbeck analogue]\label{smallenergyext}
 There exist two constants $\delta>0,C>0$ with the following property. Suppose $\psi\in W^{1,3}(\mathbb S^3,\mathbb S^3)$ such that $\|\nabla\psi\|_{L^3(\mathbb S^3)}\leq \delta$. Then there exists an extension $v\in W^{1,4}(B^4,\mathbb S^3)$ satisfying the following estimate:
$$
\|v\|_{W^{1,4}(B^4)}\leq C \|\nabla\psi\|_{L^3(\mathbb S^3)}.
$$
\end{theorem}
The following lemma will be later applied to the restriction of $u$ to a smaller ball $B_{1-\rho}$, where $u$, being harmonic, is smooth.
\begin{lemma}[interior estimate]\label{intest}
Given $u \in W^{1,4}\cap C^1 (B^4,B^4)$, there exists a constant $C$ independent of $u$ such that for half of the points $a\in B^4$ there holds
$$
\left\|\frac{1}{|u-a|}\right\|^4_{L^{4,\infty}(B^4)}\leq C\int_{B^4}|\nabla u|^4. 
$$
\end{lemma}
\begin{proof}
 By the co-area formula we have
$$
|\{x: |u(x)-a|^{-1}>\Lambda\}|=|u^{-1}(B_{\Lambda^{-1}}(a))|=\int_{B_{\Lambda^{-1}}(a)}\op{Card}(u^{-1}(x))dx\leq C\int_{B^4}|\nabla u|^4.
$$
We then observe that the measurable positive function $F_u(x):=\op{Card}(u^{-1}(x))$ belongs to $L^1(B^4)$. The maximal function $MF_u$ has $L^{1,\infty}$-norm bounded by the $L^1$-norm of $F_u$ and in particular there exists a constant $C$ independent of $u$ such that for at least half of the points $a\in B^4$ there holds
$$
\sup_\lambda\frac{1}{\lambda^4}\int_{B_\lambda(a)}F_u \leq C\int_{B^4}F_u\leq C\int_{B^4}|\nabla u|^4.
$$ 
For such $a$ we have, after the change of notation $\lambda=\Lambda^{-1}$, the wanted estimate
$$
|\{x: |u(x)-a|^{-1}>\Lambda\}|\Lambda^4\leq C\int_{B^4}|\nabla u|^4.
$$
\end{proof}
We now have the right ingredients to prove our first extension result.
\begin{theorem}[small concentration extension]\label{smallcest}
 There exists a constant $\delta\in]0,1/4[$ with the following property. For each $\phi\in W^{1,3}(\mathbb S^3,\mathbb S^3)$, such that the following local estimate holds with $\|\nabla\phi\|_{L^3(\mathbb S^3)}^3=E$:
\begin{equation}\label{eq:radiusphi}
 E(|\nabla\phi|^3,2\rho,\mathbb S^3)\leq \frac{\delta}{C_1E}.
\end{equation}
there exists a function $\tilde u\in W^{1,(4,\infty)}(B^4,\mathbb S^3)$ which equals $\phi$ on $\mathbb S^3$ in the sense of traces and satisfies
\begin{equation}\label{eq:estimatesmallenergy}
\|\nabla \tilde u\|_{L^{4,\infty}}\lesssim\frac{\|\nabla\phi\|_{L^3}^2}{\rho} +\|\nabla\phi\|_{L^3}.
\end{equation}
\end{theorem}
\begin{proof}
\textbf{Step 1.} We first observe that the harmonic extension $u$ of $\phi$ satisfies 
$$
|\nabla u|(x)\lesssim \frac{\|\phi\|_{W^{1,3}(\mathbb S^3)}}{\rho}\quad \text{ for }x\in B_{1-\rho}.
$$
A direct way to see this is by estimating via the Poisson formula together with Poincar\`e's inequality and a good covering by $\rho$-balls $B_j\subset \mathbb S^3$:

\begin{flalign*}
 |\nabla u|(x)&\lesssim \rho\left(\int_{\mathbb S^3}\frac{\nabla \phi}{|x-y|^4}dy + \int_{\mathbb S^3}\frac{|\phi|}{|x-y|^4}dy\right)&&\\
 &\lesssim\sum_{j}\frac{\fint_{B_j}|\nabla\phi|+|\phi|}{d_j^4}\rho^4,&\text{ where }d_j\sim \op{dist}(B_j,x)&\\
&\lesssim\sum_j\left(\frac{\rho}{d_j}\right)^4\fint_{B_j}|\nabla\phi|+1,&\text{ by Poincar\'e}&\\
&\lesssim\left(\sum_j (\rho/d_j)^6\right)^{2/3}\left(\sum_j\left(\fint_{B_j}|\nabla\phi|\right)^3+1\right)^{1/3},&\text{ by H\"older}&\\
&\lesssim\frac{\|\phi\|_{W^{1,3}(\mathbb S^3)}}{\rho}.&&&
\end{flalign*}

To justify the last passage we observe that $\op{Card}\{j:d_j\sim 2^j\rho\}\sim 2^{4j}$ and thus the first factor in the forelast line is bounded by $\left(\sum_{j\geq 0}2^{-2j}\right)^{2/3}$, while for the second factor of that line we use Jensen's inequality.\\

\textbf{Step 2.} We now use Lemma \ref{intest} and we observe that if $\pi_a:B^4\setminus\{a\}\to \mathbb S^3$ is the retraction of center $a$ then 
$$
|\nabla(\pi_a\circ u)|\leq C\frac{|\nabla u|}{|u-a|}.
$$
In particular using Step 1 and Lemma \ref{intest} we obtain
\begin{equation}\label{eq:intest}
 \|\nabla(\pi_a\circ u)\|_{L^{4,\infty}}\leq\|\nabla u\|_{L^\infty}\left\|\frac{1}{|u-a|}\right\|_{L^{4,\infty}}\leq C\frac{\|\nabla\phi\|_{L^3}}{\rho}\|\nabla u\|_{L^4}.
\end{equation}
\textbf{Step 3.} Consider a maximal cover $\{B_i\}$ of $\mathbb S^3=\partial B^4$ by $4$-dimensional balls  of radius $\rho$ and centers on $\partial B^4$. It is possible to find a constant $C$ depending only on the dimension such that the collection of balls of doubled radius $\{2 B_i\}$ can be written as a union of $C$ families of disjoint balls $\mathcal F_1,\ldots,\mathcal F_C$.\\

 Then apply Lemma \ref{l:fubini} to each ball $B_i\in \mathcal F_1$. This will give a new family of balls $\{B_i':\:B_i\in\mathcal F_1\}$ with radii between $\rho$ and $2\rho$ to which it will be possible to apply Lemma \ref{courant} (Courant-Lebesgue analogue). Thus $\op{dist}(u(x),\partial B^4)<\frac{1}{8}$ on $\partial (B^4\cap B_i')$ for all $B_i'$. Because of the choice of $\mathcal F_1$ it also follows that the balls $B_i'$ are disjoint.\\

If we choose the projection $\pi_a$ of Step 2 such that $\op{dist}(a,\partial B^4)>\frac{1}{4}$ then 
$$
u_1^i:=\pi_a\circ (u|_{\partial((B^4\cap B_i')})\text{ satisfies }|\nabla u_1^i|\leq C|\nabla u|\text{ on }\partial B_i'\cap B^4
$$
by the estimates of Step 2. Note that $a$ will be fixed during the whole construction.\\

We extend $u_1^i$ (denoting the extension again by $u_1^i$) inside $B_i'\cap B^4$ via Theorem \ref{smallenergyext} (Uhlenbeck analogue) obtaining a new function 
$$
u_1:=\left\{\begin{array}{lcl}
             \pi_a\circ u&\text{ on }&B^4\setminus\cup B_i',\\
             u_1^i&\text{ on }&B_i'.
            \end{array}
\right.
$$
Theorem \ref{smallenergyext} implies that $u_1$ satisfies 
$$
 \|\nabla u_1\|_{L^4(B_i')}\leq C\left(\int_{\partial B_i'}|\nabla u_1|^3\right)^{1/3}.
$$
We can rewrite this as follows:
\begin{eqnarray}
 \int_{B_i\cap B^4}|\nabla  u_1|^4&\leq& C\left(\int_{B_i\cap \partial B}|\nabla \phi|^3+\int_{\op{int}(B)\cap\partial B_i}|\nabla u_1^i|^3\right)^{4/3}\nonumber\\
&\lesssim&\left(\int_{B_i\cap \partial B}|\nabla \phi|^3\right)^{4/3}+\left(\int_{\op{int}(B)\cap\partial B_i}|\nabla u_1^i|^3\right)^{4/3}.\label{estBi}
\end{eqnarray}
We note that (using Lemma \ref{courant})
\begin{eqnarray}
\left(\int_{\partial B_i\cap\op{int}(B)}|\nabla u_1^i|^3\right)^{4/3}&\leq& \mathcal H^3(\partial B_i)^{1/3}\int_{\partial B_i\cap\op{int}(B)}|\nabla u_1^i|^4\nonumber\\
&\lesssim&\rho\int_{\partial B_i\cap\op{int}(B)}|\nabla u|^4\nonumber\\
&\lesssim&\int_{B_i\cap B^4}|\nabla u|^4\label{eq:estintbi}
\end{eqnarray}
therefore $u_1$ still satisfies \eqref{equiintest} with a constant $C_1$ which is now changed by a universal factor.\\

\textbf{Step 4.} It is possible to repeat the same operation starting from the function $u_1$ and using the balls of the family $\mathcal F_2$ to obtain a function $u_2$, and then do the same iteratively for all the families $\mathcal F_2,\ldots,\mathcal F_C$.\\

Denote by $\mathcal R$ the union of all the perturbed balls $B_i'$ corresponding to the families $\mathcal F_1,\ldots,\mathcal F_C$. Recall that the number of families is equal to the maximal number of overlaps of balls of different families, and depends only on the dimension. Then iterating the estimates \eqref{estBi} using \eqref{eq:estintbi} for all families $\mathcal F_i$ we obtain for the last function $u_C$
\begin{eqnarray}\label{eq:estnearbdry}
 \int_{\mathcal R}|\nabla u_C|^4&\lesssim& E(|\nabla \phi|^3,2\rho,\mathbb S^3)^{1/3}\sum_i\int_{B_i\cap\partial B}|\nabla \phi|^3 +\int_{\mathcal R}|\nabla u|^4\nonumber\\
&\leq& \|\nabla \phi\|_{L^3(\mathbb S^3)}^3\left(E(|\nabla \phi|^3,2\rho,\mathbb S^3)^{1/3} + \|\nabla\phi\|_{L^3(\mathbb S^3)}\right),
\end{eqnarray}
where for the last inequality we also used the elliptic estimates for $u$ in terms of $\phi$.\\

\textbf{Step 5.} We now collect the estimate \eqref{eq:intest} for the part $B\setminus \mathcal R\subset B_{1-\rho}$ and \eqref{eq:estnearbdry}. Observe that in general $\|f\|_{L^{4,\infty}}\lesssim\|f\|_{L^4}$ and that the $L^{4,\infty}$-norm satisfies the triangle inequality. We obtain
\begin{equation}\label{eq:estimatesmallenergy2}
\|\nabla \tilde u\|_{L^{4,\infty}}\lesssim\frac{\|\nabla\phi\|_{L^3}^2}{\rho} +\|\nabla\phi\|_{L^3} + \|\nabla\phi\|_{L^3}^{3/4}E(|\nabla\phi|^3,2\rho,\mathbb S^3)^{1/12}.
\end{equation}
Using the trivial estimate $E(|\nabla \phi|^3,2\rho,\mathbb S^3)\leq \int_{\mathbb S^3}|\nabla\phi|^3$, the wanted estimate follows.
\end{proof}
\subsection{The case of large energy concentration}\label{sec:finalstep}

In this section $E$ will denote an \emph{upper bound} for the $L^3$-energy of boundary value functions $\phi$. Following Theorem \ref{smallcest} we are led to divide the set of boundary value functions $W^{1,3}(\mathbb S^3, \mathbb S^3)$ into two classes, based on whether or not the energy concentrates. We will do the division based on the following parameters: the energy bound $E$, a concentration radius $\rho_E$ and an upper bound on the concentration $A_E$. $\rho_E,A_E$ will be fixed in Section \ref{sec:endofproof}, depending only on $E$. 
We introduce the following two classes of ``good'' and ``bad'' boundary value functions:
\begin{equation}\label{eq:goodbad}
\begin{array}{l} \mathcal G^E:=\{\phi\in W^{1,3}(\mathbb S^3,\mathbb S^3):\: \|\nabla\phi\|_{L^3}^3\leq E, E_\phi \leq A_E\},\\
 \mathcal B^E:=\{\phi\in W^{1,3}(\mathbb S^3,\mathbb S^3):\: \|\nabla\phi\|_{L^3}^3\leq E, E_\phi >A_E\}.
\end{array}
\end{equation}
where
$$
E_\phi:=E(|\nabla\phi|^3,\rho_E, \mathbb S^3)\text{ for }\phi\in W^{1,3}(\mathbb S^3,\mathbb S^3).
$$
The precise steps of our extension construction are as follows (see also the scheme \eqref{summary}):
\begin{enumerate}
 \item Theorem \ref{smallcest} gives a good estimate for the boundary values in $\mathcal G^E$.
\item If $\phi\in\mathcal B^E$ has average close to zero, i.e. 
$$\left|\int_{\mathbb S^3}\phi\right|\leq \frac{1}{4},$$ 
then it is possible to write $\phi=\phi_1\phi_2$ with 
$$\int_{\mathbb S^3} |\nabla \phi_i|^3\leq E-A_E/2$$
(the product of $\mathbb S^3$-valued functions is pointwise the product on $\mathbb S^3\simeq SU(2)$).
\item If we are not in the two cases above, we use the functions 
$$F_v(x):= -v +(1-|v|^2)(x^*-v)^*$$
where $a^*=\frac{a}{|a|^2}, v\in B^4$, which form a subset of the M\"obius group of $B^4$. We have two cases:
\begin{enumerate}
\item $\forall v\in B^4\text{ there holds }\left|\int_{\mathbb S^3}\phi\circ F_v\right|>\frac{1}{4},$
in which case 
$$\tilde u(v):=\pi_{\mathbb S^3}\left(\int_{\mathbb S^3}\phi\circ F_v\right)$$
gives an extension of $\phi$ with values in $\mathbb S^3$ and satisfying 
$$\|u\|_{W^{1,4}}\lesssim\|\phi\|_{W^{1,3}},$$
\item $\exists v\in B^4\text{ such that }\left|\int_{\mathbb S^3}\phi\circ F_v\right|\leq\frac{1}{4},$
in which case we can apply the reasoning of cases (1), (2) above to $\tilde\phi:=\phi\circ F_v$. Since $F_v$ is conformal and $|\phi|=|\tilde\phi|=1$ we have 
$$\|\nabla \phi\|_{L^3}=\|\nabla \tilde \phi\|_{L^3},\quad\|\phi\|_{W^{1,3}}=\|\tilde\phi\|_{W^{1,3}}.$$ 
Again we reason differently in the two cases $\tilde\phi\in \mathcal G^E$ and $\tilde\phi\in \mathcal B^E$.
\end{enumerate}
\item If in case (3b) $\tilde\phi\in \mathcal B^E$ then we apply case (2) to $\tilde\phi$ and we can express 
$$\tilde\phi=\tilde\phi_1\tilde\phi_2$$ 
and 
$$\phi=(\tilde\phi_1\circ F_v^{-1})(\tilde\phi_2\circ F_v^{-1}).$$ 
Then $\phi_i:=\tilde\phi_i\circ F_v^{-1}$ are as in case (2).
\item If in case (3b) $\tilde\phi\in \mathcal G^E$ then we apply case (1) to $\tilde \phi$. With a careful study of the relation between the position of $v\in B^4$ relative to $\partial B^4$ and the parameter $\rho_E$, we construct 
$$u\in W^{1,(4,\infty)}(B^4, \mathbb S^3)\text{ extending }\phi=\tilde\phi\circ F_v^{-1}$$
starting from the extension $\tilde u$ of $\tilde\phi$ given in case (1). 
\end{enumerate}
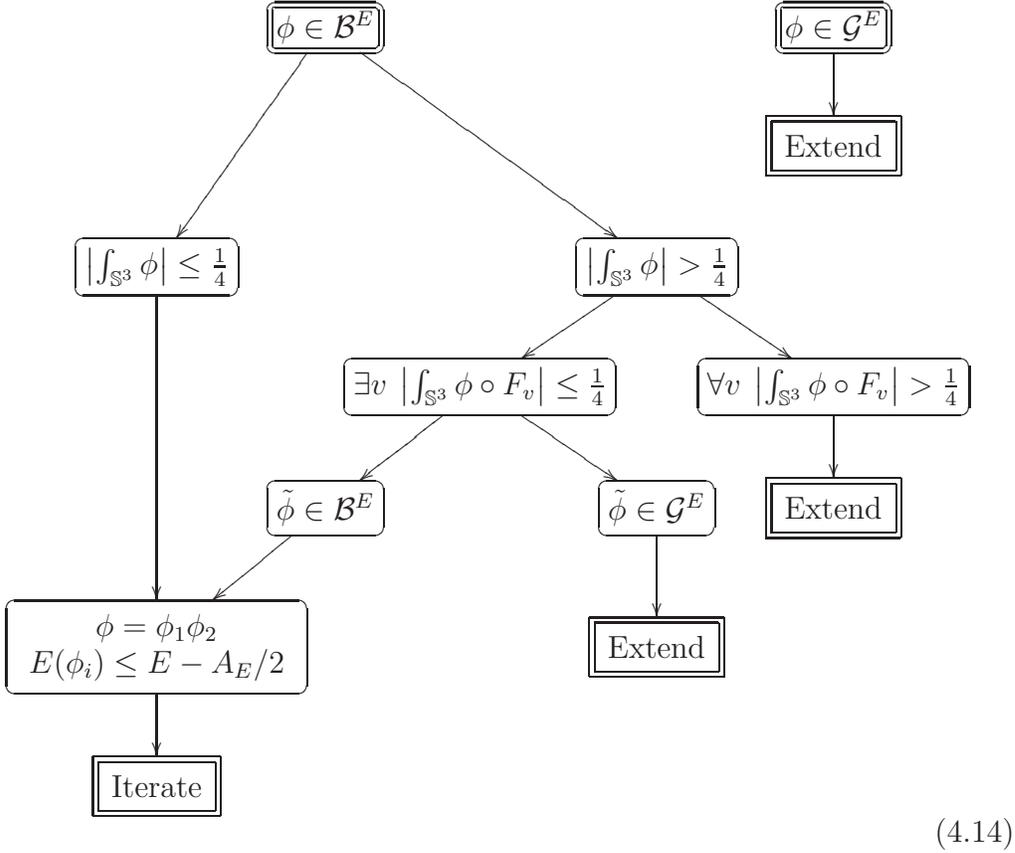
\begin{equation}\label{summary}
\xymatrix@C=-15pt{
&*+[F=:<3pt>]{\phi\in\mathcal B^E}\ar[ddl]\ar[ddrr]&&&*+[F=:<3pt>]{\phi\in\mathcal G^E}\ar[d]\\
&&&&*++[F=]{\txt{Extend}}\\
*+[F-:<3pt>]{\left|\int_{\mathbb S^3}\phi\right|\leq\frac{1}{4}}\ar[ddd]&&&*+[F-:<3pt>]{\left|\int_{\mathbb S^3}\phi\right|>\frac{1}{4}}\ar[dl]\ar[dr]\\
&&*+[F-:<3pt>]{\exists v\:\left|\int_{\mathbb S^3}\phi\circ F_v\right|\leq\frac{1}{4}}\ar[dl]\ar[dr]&&*+[F-:<3pt>]{\forall v\:\left|\int_{\mathbb S^3}\phi\circ F_v\right|>\frac{1}{4}}\ar[d]\\
&*+[F-:<3pt>]{\tilde\phi\in \mathcal B^E}\ar[dl]&&*+[F-:<3pt>]{\tilde\phi\in \mathcal G^E}\ar[d]&*++[F=]{\txt{Extend}}\\
*+[F-:<3pt>]{\begin{array}{ccc}\phi=\phi_1\phi_2\\E(\phi_i)\leq E-A_E/2\end{array}}\ar[d]&&&*++[F=]{\txt{Extend}}\\
*++[F=]{\txt{Iterate}}
}
\end{equation}

\begin{proposition}[balancing $\Rightarrow$ splitting]\label{zeroaverage}
There exists a geometric constant $C$ with the following property. Suppose that $\phi\in\mathcal B^E$ with the notations of \eqref{eq:goodbad}, and assume $A_E\leq 1/C$ and $\rho_E\leq e^{-C\max\{EA_E,(EA_E)^3\}}$. Further assume that as a function in $W^{1,3}(\mathbb S^3,\mathbb R^4)$, $\phi$ satisfies
$$
\left|\fint_{\mathbb S^3}\phi\right|\leq \frac{1}{4}.
$$
Then identifying $\mathbb S^3\sim SU(2)$ there exists a decomposition
\begin{equation}\label{decomp}
\phi=\phi_1\phi_2
\end{equation}
such that for both $i=1,2$ we have that 
\begin{equation}\label{energybootstrap}
 \int_{\mathbb S^3}|\nabla\phi_i|^3<E-A_E/2.
\end{equation}
\end{proposition}
\begin{proof}
 We will proceed through several steps.\\
\textbf{Step 1.} Fix a concentration ball $B=B^{\mathbb S^3}(\rho_E,x_0)$ such that 
\begin{equation}\label{concentrball}
\int_B|\nabla\phi|^3>A_E.
\end{equation}
\textbf{Step 2.} Consider dyadic rings in $\mathbb S^3$ defined as $R_i:=2^{i+1}B\setminus 2^iB$ where we denote $2^iB=B^{\mathbb S^3}(2^i\rho_E,x_0)$. We observe that for $N_E<-C\log_2\rho_E$ the rings with $i\leq N_E$ stay all disjoint (we will fix $N_E$ later). Therefore there holds
$$
 \sum_{i=1}^{N_E}\int_{R_i}|\nabla\phi|^3<E.
$$
By pigeonhole principle, there exists $i_0\in\{1,\ldots,N_E\}$ such that
$$
\int_{R_{i_0}}|\nabla\phi|^3<\frac{E}{N_E}.
$$
Again by pigeonhole principle (using the fact that the cubes are dyadic) there exists then $t\in[2^{i_0+1}\rho_E,2^{i_0}\rho_E]$ such that 
\begin{equation}\label{ringsest}
 t\int_{\partial B^{\mathbb S^3}(t,x_0)}|\nabla\phi|^3<C\frac{E}{N_E},
\end{equation}
where $C$ is a constant depending only on the geometry of $\mathbb S^3$.\\
\textbf{Step 3.} Denote $B_t=B^{\mathbb S^3}(t,x_0)$ as in Step 2. We define the function $\tilde\phi_1$ via a suitable harmonic extension outside of $B_t$ as follows:
\begin{equation*}
 \left\{\begin{array}{lll}
\tilde\phi_1=\phi&\text{ on }&\partial B_t,\\
\Delta(\tilde \phi_1\circ \Psi)=0&\text{ on }&B_1^{\mathbb R^3},         
        \end{array}
\right.
\end{equation*}
where $\Psi:\mathbb R^3\to \mathbb S^3\setminus\{x_0\}$ is a stereographic projection composed with a dilation of $\mathbb R^3$, such that $\Psi(B^{\mathbb R^3}(1,0))=\mathbb S^3\setminus B_t$. On $B_t$ we define $\tilde\phi_1\equiv \phi$. By H\"older's inequality, using elliptic estimates and the conformality of dilations and inverse stereographic projections, we have
\begin{equation}\label{harmextest}
\begin{array}{rcl} t\int_{\partial B_t}|\nabla\tilde\phi_1|^3&\geq& C\left(\int_{\partial B_t}|\nabla\tilde\phi_1|^2\right)^{3/2}=C\left(\int_{\partial B_1^{\mathbb R^3}}|\nabla\tilde\phi_1\circ\Psi|^2\right)^{3/2}\\
&\geq& C\int_{B^{\mathbb R^3}_1}|\nabla\tilde\phi_1\circ \Psi|^3=C\int_{\mathbb S^3\setminus B_t}|\nabla\tilde\phi_1|^3.
\end{array}
\end{equation}
However, note that in general $\tilde\phi_1$ will have values in $\mathbb R^4$ but we can insure that they belong to $\mathbb S^3$ only on the ball $B_t$.\\
\textbf{Step 4.} We define then 
$$
\phi_1=\pi_{\mathbb S^3}\circ\tilde\phi_1.
$$ 
We claim that \emph{if $N_E$ is large enough then $\phi_1$ satisfies some estimates like \eqref{harmextest} where the constants $C$ are worsened just by a factor close to $1$}. Indeed, \eqref{ringsest} together with the Sobolev embedding $W^{1,3}\to C^{0,1/3}$ (valid for $2$-dimensional domains like $\partial B_t$) implies that $\phi|_{\partial B_t}$ stays close to a fixed point of $\mathbb S^3$ as in the proof of lemma \ref{courant}(Courant-Lebesgue analogue). Therefore also $\phi_1\circ \Psi|_{\partial B_1^{\mathbb R^3}}$ does. By mean value theorem, $\phi_1\circ \Psi|_{B_1^{\mathbb R^3}}$ and thus $\tilde\phi_1|_{B_t}$ will not have a larger distance to the same point of $\mathbb S^3$. Quantitatively, there exists a geometric constant $C$ such that if
\begin{equation}\label{condNE}
 \frac{E}{N_E}\leq C
\end{equation}
then
$$
\op{dist}(\tilde\phi_1,\mathbb S^3)\leq 1/2.
$$
This implies via the pointwise bound
$$
|\nabla(\pi_{\mathbb S^3}\circ f)|\leq C\frac{|\nabla f|}{|f|}
$$
that pointwise a.e. there holds the following estimate
$$
|\nabla\phi_1|\leq C|\nabla\tilde\phi_1|,
$$
which proves our claim. This claim together with the estimates \eqref{harmextest} and \eqref{ringsest} implies the following bound, valid under condition \eqref{condNE}:
\begin{equation}\label{phi1est}
 \int_{\mathbb S^3\setminus B_t}|\nabla\phi_1|^3\leq C\frac{E}{N_E}.
\end{equation}
\textbf{Step 5.} We now estimate from below the energy of $\phi|_{\mathbb S^3\setminus B_t}$. Denote by $\bar\phi_\Omega$ the average of $\phi$ on a domain $\Omega\subset \mathbb S^3$. First we use the Poincar\'e inequality on $\mathbb S^3\setminus B_t$ and the fact that $|\phi|\equiv 1$ almost everywhere.
\begin{equation}\label{poincare}
\begin{split}
 \int_{\mathbb S^3\setminus B_t}|\nabla\phi|^3&\gtrsim\int_{\mathbb S^3\setminus B_t}|\phi - \bar\phi_{\mathbb S^3\setminus B_t}|^3\gtrsim\left(\int_{\mathbb S^3\setminus B_t}|\phi - \bar\phi_{\mathbb S^3\setminus B_t}|\right)^3\\
 &\gtrsim\left(|\mathbb S^3\setminus B_t|(1-|\bar\phi_{\mathbb S^3\setminus B_t}|)\right)^3.
\end{split}
\end{equation}
Using the fact that $|\bar \phi_{\mathbb S^3}|\leq \frac{1}{4}$ and the triangle inequality we have 
\begin{equation}\label{averages}
|\mathbb S^3\setminus B_t||\bar\phi_{\mathbb S^3\setminus B_t}| \leq \frac{1}{4}|\mathbb S^3| +|B_t||\bar\phi_{B_t}|.
\end{equation}
\eqref{poincare} and \eqref{averages} together with the estimate $|\bar\phi_{B_t}|\leq 1$ give
\begin{equation}\label{zeroavest}
 \int_{\mathbb S^3\setminus B_t}|\nabla\phi|^3\geq \frac{1}{C}\left(\frac{3}{4}|\mathbb S^3| - 2|B_t|\right)^3.
\end{equation}
From this inequality and since we assumed $A_E$ to be small, we obtain
\begin{equation}\label{est1phi}
 \int_{\mathbb S^3\setminus B_t}|\nabla\phi|^3\geq A_E\quad \text{ if }\quad t<C,
\end{equation}
for some geometric constant $C$.\\
\textbf{Step 6.} We now define $\phi_2:=\phi_1^{-1}\phi$ where the pointwise product uses the group operation on $\mathbb S^3\sim SU(2)$. Observe that since $|\phi|=|\phi_1|=1$ a.e., 
$$
|\nabla(\phi_1^{-1}\phi)|=|\phi^{-1}\nabla\phi_1\phi_1^{-1}\phi +\phi_1^{-1}\nabla\phi|\leq|\nabla\phi|+|\nabla\phi_1|.
$$
We then apply this last inequality together with H\"older's inequality to obtain that if the number of rings $N_E$ in \eqref{phi1est} is so large that $\|\nabla\phi_1\|_{L^3(\mathbb S^3\setminus B_t)}\leq \|\nabla\phi\|_{L^3(\mathbb S^3\setminus B_t)}$  then
$$
 \int_{\mathbb S^3\setminus B_t}|\nabla\phi_2|^3\leq \int_{\mathbb S^3\setminus B_t}|\nabla\phi|^3+7\left(\int_{\mathbb S^3\setminus B_t}|\nabla\phi_1|^3\right)^{\frac{1}{3}}\left(\int_{\mathbb S^3\setminus B_t}|\nabla\phi|^3\right)^{\frac{2}{3}}.
$$
By using \eqref{est1phi} and \eqref{phi1est} we then obtain (under the hypotheses \eqref{condNE} and $A_E\leq 1/C$ needed for these inequalities to hold)
\begin{equation}\label{estphi2}
 \int_{\mathbb S^3\setminus B_t}|\nabla\phi_2|^3\leq \int_{\mathbb S^3\setminus B_t}|\nabla\phi|^3+C\frac{E}{N_E^{\frac{1}{3}}}\leq E- A_E +C\frac{E}{N_E^{\frac{1}{3}}}.
\end{equation}
\textbf{Step 7.} It is now possible to conclude. The estimate \eqref{energybootstrap} for $\phi_2$ follows from \eqref{estphi2} and \eqref{concentrball}, if the last summand in \eqref{estphi2} is smaller than $A_E/2$. This requirement translates into
\begin{equation}\label{reqne1}
N_E\geq CE^3 A_E^3.
\end{equation}
The estimate \eqref{energybootstrap} for $\phi_1$ follows by observing that by construction $\phi_1\equiv\phi$ on $B_t$. It follows from \eqref{est1phi} and \eqref{phi1est} that
$$
\int_{\mathbb S^3}|\nabla\phi_1|^3=\int_{B_t}|\nabla\phi|^3 +\int_{\mathbb S^3\setminus B_t}|\nabla\phi_1|^3\leq E-A_E+C\frac{E}{N_E}.
$$
Therefore the request that the last term is $\leq E-A_E/2$ translates into
\begin{equation}\label{reqne2}
N_E\geq CE A_E.
\end{equation}
Recall that in Step 2 we connected $N_E$ to $\rho_E$ by the condition $N_E<-C\log_2\rho_E$, so \eqref{reqne1}, \eqref{reqne2} translate into the requirement $\rho_E\leq e^{-C\max\{E A_E, (E A_E)^3\}}$ assumed in the thesis. The requirement on $A_E$ was needed for the reasoning of Step 5.
\end{proof}

\begin{rmk}
 The proof of \eqref{zeroavest} in Step 5 gives the following general estimate valid for bounded Sobolev functions on a compact manifold $M$ and for any Poincar\'e domain $\Omega\subset M$:
\begin{equation}\label{generalest}
\|\nabla\phi\|_{L^p(\Omega)}\geq C_{\Omega}\left[|M|(\|\phi\|_{L^\infty(M)}-|\bar\phi_M|)-2\|\phi\|_{L^\infty(M)}|M\setminus\Omega|\right],
\end{equation}
where $C_\Omega$ is the Poincar\'e constant of $\Omega$.
\end{rmk}

Consider now the following conformal transformations of the unit ball $B^4$:
$$
F_v(x)=-v+(1-|v|^2)(x^* -v)^*\text{, where }v\in B^4\text{ and }a^*=\frac{a}{|a|^2}.
$$
We want to prove here the following proposition:
\begin{proposition}[balancing $\Rightarrow$ extension]\label{prop:moebius}
Let $\phi\in W^{1,3}(\mathbb S^3,\mathbb S^3)$. Suppose that for all $v\in B^4$ there holds 
\begin{equation}\label{eq:bigav}
 \left|\fint_{\mathbb S^3}\phi\circ F_v\right|\geq \frac{1}{4}.
\end{equation}
Then the following function $u:B^4\to \mathbb S^3$ extends $\phi$ 
\begin{equation}\label{eq:biharmext}
 u(v):=\pi_{\mathbb S^3}\left(\fint_{\mathbb S^3}\phi\circ F_v\right),\text{ where }\: \pi_{\mathbb S^3}(a)=\frac{a}{|a|}\text{ for }a\in \mathbb R^4\setminus\{0\}.
\end{equation}
Moreover, there exists a constant $C$ independent of $\phi$ such that the following estimate holds:
\begin{equation}\label{eq:biharmextest}
 \|\nabla u\|_{L^4(B^4)}\leq C\|\nabla\phi\|_{L^3(\mathbb S^3)}.
\end{equation}
\end{proposition}
\begin{proof}
\textbf{Step 1.} We note that after a change of variable there holds
$$
\fint_{\mathbb S^3}\phi\circ F_v(x)dx=\fint_{\mathbb S^3}\phi(y) |(F^{-1}_v)'|^3(y) dy.
$$
where $|(F^{-1}_v)'|$ is the conformal factor of $DF_v^{-1}$. We know from Lemma \ref{lem:basicpropfv} that 
$$
|(F^{-1}_v)'|(y)=|F_{-v}'|(y)=\frac{1-|v|^2}{|y+v|^2},
$$
therefore
$$
\fint_{\mathbb S^3}\phi\circ F_v=\fint_{\mathbb S^3}\phi(y) \left(\frac{1-|v|^2}{|y+v|^2}\right)^3dy.
$$
As follows from \cite{nicolesco}, in dimension $4$ the function 
\[
 K(x,y)=|\mathbb S^3|^{-1}\left[\frac{1-|y|^2}{|x-y|^2}\right]^3 
\]
is the Poisson kernel for the equation
\begin{equation}\label{eq:biharm}
 \left\{\begin{array}{ll}
         \Delta^2u=0&\text{ on }B^4,\\
         \left.\frac{\partial u}{\partial \nu}\right|_{\partial B^4}=0,& u|_{\partial B}=\phi.
        \end{array}
\right.
\end{equation}
Therefore the function 
$$
\tilde u(v):=\fint_{\mathbb S^3} \phi\circ F_v
$$
is equal to the biharmonic extension of $\phi$ given by equation \eqref{eq:biharm}.\\
\textbf{Step 2.} We recall the following classical estimate which holds for equation \eqref{eq:biharm}:
$$
\|\nabla u\|_{L^4(B^4)}\leq C\|\nabla\phi\|_{L^3(B^3)}.
$$
For the proof of this estimate see \cite{GGS}, where the stronger and more natural estimate $\|u\|_{W^{1,4}(\Omega)}\leq \|\phi\|_{W^{1-1/4,4}(\partial \Omega)}$ is obtained in Chapter 2.\\
\textbf{Step 3.} We note that 
$$
\forall v\in B^4,\quad 1/4\leq|\tilde u(x)|\leq C
$$
because of our hypothesis \eqref{eq:bigav}, $|\phi|\equiv 1$ and by the elementary estimate $\int_{\mathbb S^3}\left(\frac{1-|v|^2}{|y+v|^2}\right)^3dy\leq C$. As in Step 2 of the proof of Theorem \ref{smallcest} (in the present case we have $\pi_{\mathbb S^3}=\pi_a$ for $a=0$) we then obtain the pointwise estimate
$$
|\nabla(\pi_{\mathbb S^3} \circ\tilde u)|\sim|\nabla \tilde u|.
$$
From this and Step 2 the estimate \eqref{eq:biharmextest} follows.
\end{proof}

We next consider the case in which the hypothesis of Proposition \ref{prop:moebius}(balancing $\Rightarrow$ extension) is false, i.e. that 
\begin{equation}\label{eq:avsmall}
\exists v\in B^4\quad \left|\fint_{\mathbb S^3}\phi\circ F_v\right|\leq \frac{1}{4}.
\end{equation}
We then denote 
\begin{equation}\label{eq:tildephi}
 \tilde\phi:=\phi\circ F_v\text{ for a fixed }v\text{ satisfying \eqref{eq:avsmall}.}
\end{equation}
Note that $F_v|_{\mathbb S^3}$ is conformal and bijective (see Section \ref{sec:moeblemmas}) and thus for $A\subset \mathbb S^3$
$$
\int_A|\nabla\tilde \phi|^3=\int_{F_v^{-1}(A)}|\nabla\phi|^3,
$$
in particular $\tilde\phi$ has the same energy bound $E$ as $\phi$ (we use here the notation of \eqref{eq:goodbad}). We start with an easy result:
\begin{lemma}\label{lem:tphibad}
Under the assumption \eqref{eq:avsmall} and with the notation \eqref{eq:tildephi}, suppose that $\tilde\phi\in\mathcal B^E$. Then there exist $\phi_1,\phi_2\in W^{1,3}(\mathbb S^3,\mathbb S^3\simeq SU(2))$ such that
$$
\phi=\phi_1\phi_2,\quad\int_{\mathbb S^3}|\nabla\phi_i|^3\leq E-A_E/2\text{ for }i=1,2,
$$
with the constant $A_E$ coming from Proposition \ref{zeroaverage}(balancing $\Rightarrow$ splitting).
\end{lemma}
\begin{proof}
 We observe that Proposition \ref{zeroaverage} applies to $\tilde\phi$ directly, due to our hypotheses. Therefore we can find $\tilde \phi_1,\tilde \phi_2\in W^{1,3}(\mathbb S^3, SU(2))$ such that 
$$
\tilde\phi=\tilde\phi_1\tilde\phi_2,\quad\int_{\mathbb S^3}|\nabla\tilde\phi_i|^3\leq E-A_E/\text{ for }i=1,2.
$$
We then precompose with $F_v^{-1}$ which preserves the pointwise product and the $L^3$-energy of the gradients,2 obtaining the same decomposition for $\phi$.
\end{proof}
The case $\tilde\phi\in\mathcal G^E$ is a bit more difficult:
\begin{proposition}\label{prop:tphigood}
Under the assumption \eqref{eq:avsmall} and with the notation \eqref{eq:tildephi}, suppose that $\tilde\phi\in\mathcal G^E$. Then there exists an extension $u\in W^{1,(4,\infty)}(B^4, \mathbb S^3)$ of $\phi$ such that
\begin{equation}\label{eq:estphigood}
 \|\nabla u\|_{L^{4,\infty}(B^4)}\leq \frac{C}{\rho_E}\|\nabla \phi\|_{L^3(\mathbb S^3)}^2 +\|\nabla\phi\|_{L^3(\mathbb S^3)},
\end{equation}
under the assumption that
\begin{equation}\label{eq:vsmall}
\rho_E\leq\frac{1}{4}.
\end{equation}
\end{proposition}
\begin{proof}
To simplify notations $\rho=\rho_E$ during this proof. We divide the domain $B^4$ into 
$$
A:=F_v^{-1}(B(0,1-\rho)),\quad A':=B^4\setminus A.
$$
Using Lemma \ref{lem:morepropfv} it follows that there exists a geometric constant $C$ and a function $h(v)$ such that for $x\in A$ and under the condition \eqref{eq:vsmall},
\begin{equation}\label{eq:estfvprime}
 \frac{h(v)}{C}\leq |F_v'|(x)\leq Ch(v).
\end{equation}
We can use \eqref{eq:estfvprime} to control the $L^{4,\infty}$-norm of $\nabla u$ restricted to $A$ via the similar norm of $\nabla\tilde u$:
\begin{eqnarray*}
 \left|\left\{x\in A:|\nabla u|(x)>\Lambda\right\}\right|&=&\left|\left\{x\in A:|\nabla \tilde u|(F_v(x))|F_v'|(x)>\Lambda\right\}\right|\\  
&\leq&\left|\left\{x\in A:|\nabla \tilde u|(F_v(x))>\Lambda/(Ch(v))\right\}\right|\\
&=&\int_{F_v(A)\cap\{y:|\nabla\tilde u|(y)>\Lambda/(Ch(v))\}}|F_v'|^{-4}dy\\
&\leq&C^4h^{-4}(v)\left|\left\{y\in B_{1-\rho}:|\nabla\tilde u|>\Lambda/(Ch(v))\right\}\right|\\
&\leq&C^8\Lambda^{-4}\|\nabla\tilde u\|_{L^{4,\infty}(B_{1-\rho})}^4.
\end{eqnarray*}
By bringing $\Lambda$ to the other side it follows that 
\begin{equation}\label{eq:esttua}
\Lambda^4\left|\left\{x\in A:|\nabla u|(x)>\Lambda\right\}\right|\leq C^8\|\nabla\tilde u\|_{L^{4,\infty}(B(0,1-\rho))}.
\end{equation}
On the other hand we can use the conformal invariance, the invertibility of $F_v$ and the usual estimate between $L^{4,\infty}$ and $L^4$ to complete a first step of the proof:
\begin{equation}\label{eq:esttuap}
 \Lambda^4|\{x\in A':\:|\nabla u|(x)>\Lambda\}|\leq C\|\nabla u\|_{L^4(A')}^4=C\|\nabla\tilde u\|_{L^4(B\setminus B_{1-\rho})}).
\end{equation}
We now sum \eqref{eq:esttua} to \eqref{eq:esttuap} and we take the supremum on $\Lambda>0$. It follows that up to increasing $C$,
\begin{equation}\label{eq:esttu}
 [\nabla u]_{L^{4,\infty}(B^4)}\leq C(\|\nabla\tilde u\|_{L^{4,\infty}(B_{1-\rho})}+\|\nabla\tilde u\|_{L^4(B\setminus B_{1-\rho})}).
\end{equation}
The estimate \eqref{eq:esttu} together with Theorem \ref{smallcest} applied to $\tilde u$ gives the wanted estimate for the first summand, while for the second summand we proceed as in Step 3 of the proof of Theorem \ref{smallcest}. We use the small concentration regions $B_i$ for $\tilde \phi$, on which we apply the Courant lemma \ref{courant} which allows to project the values of $u:=\tilde u\circ F_v^{-1}$ as well on $\mathbb S^3$, with little change of the gradient of $u$. We observe that $F_v^{-1}$ is conformal, so the $L^3$-energy of $\tilde u$ on $\partial B_i$ is the same as the $L^3$-energy of $u$ on $\partial F_v^{-1}(B_i)$ and use the Uhlenbeck extension result of Theorem \ref{smallenergyext}(Uhlenbeck analogue) for $\tilde u$ as in Step 3 of the proof of Theorem \ref{smallcest}. We obtain:
$$
\|\nabla u\|_{L^4(F_v^{-1}(B\setminus B_{1-\rho})}=\|\nabla \tilde u\|_{L^4(B\setminus B_{1-\rho})}\leq C\|\nabla\tilde\phi\|_{L^3(\mathbb S^3)}= C\|\nabla\phi\|_{L^3(\mathbb S^3)}.
$$
This and \eqref{eq:esttu} conclude the proof.
\end{proof}

\subsection{End of the proof of Theorem B''}\label{sec:endofproof}
We will refer to the scheme \eqref{summary} for the idea of the proof. \\

\textbf{Choice of $A_E$.} In \eqref{eq:goodbad} take $A_E\leq \frac{\delta}{C_1}$ with the notations of Theorem \ref{smallcest} so that it applies to give extensions for the small concentration case (``good'' boundary conditions). Here $\delta$ is the constant coming from the Uhlenbeck procedure on regions of radius $\rho_E$ near $\partial B^4$. If necessary diminish $A_E$ such that the requirement $A_E\leq C^{-1}$ of Proposition \ref{zeroaverage}(balancing $\Rightarrow$ splitting) is also satisfied. \\

\textbf{Choice of $\rho_E$.} Recall that the constant $C$ appearing there was depending just on the volume of $\mathbb S^3$. For the radius of concentration $\rho_E$ we need just to impose the bound present in Proposition \ref{zeroaverage}, which with the choices of 
$A_E$ just done becomes $\rho_E\lesssim e^{-C\max(1, E^3)}$. \\

\textbf{Estimates for extensions.} Consider again the scheme \eqref{summary}. Each time we extend some boundary datum $\phi$ obtained during our constructions via a function $u:B^4\to \mathbb S^3$, we do so with one of the following estimates:
\begin{itemize}
 \item In the case of the extensions of Theorem \ref{smallcest} or of Proposition \ref{prop:tphigood} (which in turn actually depends on Theorem \ref{smallcest}) we have
$$\|\nabla u\|_{L^{4,\infty}}\lesssim\frac{\|\nabla\phi\|_{L^3}^2}{\rho_E} +\|\nabla\phi\|_{L^3}.$$
 \item In the case of the biharmonic extension of Proposition \ref{prop:moebius}(balancing $\Rightarrow$ extension) we have the much better
$$\|\nabla u\|_{L^4}\lesssim\|\nabla\phi\|_{L^3}.$$
\end{itemize}
The number of iterations to be made when we apply the procedure described in scheme \eqref{summary} is bounded by 
$$
E\left/\frac{A_E}{2}\right.\sim E^2.
$$
Since each iteration creates two new boundary value functions out of one, in the end we may have a decomposition into no more than 
$$
e^{CE^2}\text{ boundary value functions.}
$$
By the triangle inequality we see that in this case there exists an extension of the initial $\phi$ satisfying
\begin{equation}\label{eq:finalestimate}
 \|\nabla u\|_{L^{4,\infty}}\lesssim e^{C\|\nabla\phi\|_{L^3}^9}\|\nabla\phi\|_{L^3}^2 +e^{C\|\nabla\phi\|_{L^3}^6}\|\nabla\phi\|_{L^3}.
\end{equation}
this gives the estimate \eqref{eq:estimate} of Theorem B'', finishing the proof. $\square$

\section{Controlled global gauges}\label{sec:controlledcoul}
\noindent
We now fix a closed Riemannian $4$-manifold $(M,h)$ with a connection $A\in W^{1,2}(\wedge^1M,su(2))$ whose curvature will be denoted by $F$. We want to find a global gauge for $A$ in which $\|A\|_{W^{1,(4,\infty)}}\leq f(E)$ where $E:=\int_M|F|^2$.\\

We will use the following two results. The first one is the restatement of Theorem B' which we repeat for easier reference.
\begin{theoremB'}
 Fix a trivial $SU(2)$-bundle $E$ over the ball $B^4$. There exists a function $f_1:\mathbb R^+\to\mathbb R^+$ with the following property. If $g\in W^{1,3}(\mathbb S^3,SU(2))$ gives a trivialization of the restricted bundle $E|_{\partial B^4}$, then there exists an extension of $g$ to a trivialization $\tilde g\in W^{1,(4,\infty)}(B^4,SU(2))$ such that the following estimate holds:
\begin{equation*}
 \|\nabla \tilde g\|_{L^{4,\infty}(B^4)}\leq f_1\left(\|\nabla g\|_{L^3(\mathbb S^3)}\right).
\end{equation*}
\end{theoremB'}
The second theorem is the main result of \cite{Uhl2}.
\begin{theorem}[Uhlenbeck gauge]\label{uhlenbeckext}
 There exists $\epsilon_0>0$ such that if the curvature satisfies $\int_{B_1}|F|^2\leq\epsilon_0$ then there exists a gauge $\phi\in W^{2,2}(B_1,SU(2))$ such that in that gauge the connection satisfies $\|A_\phi\|_{W^{1,2}(B_1)}\leq C\|F\|_{L^2(B_1)}$ with $C>0$ depending only on the dimension.
\end{theorem}

\begin{theorem}\label{ggest}
For each closed boundaryless $4$-manifold $M^4$ there exists a function $f:\mathbb R^+\to\mathbb R^+$ with the following properties.\\
 Let $\nabla$ be a $W^{1,2}$ connection for an $SU(2)$-bundle over $M$. Then there exists a \underbar{global}
$W^{1,(4,\infty)}$ section of the bundle over the whole $M^4$ such that in the corresponding
trivialization $\nabla$ is given by $d+A$ with the following bound.

\begin{equation}\label{ggesteq}
\|A\|_{L^{(4,\infty)}}\leq f\left(\|F\|_{L^2(M)}\right),
\end{equation}
where $F$ is the curvature form of $\nabla$.
\end{theorem}
\subsection{Scheme of the proof}
We indicate here the sketch of the proof, before going through the details.

\begin{proof}
We will denote the $L^2$-norm of $F$ by $E$. We may assume that a first guess for $A$ (i.e. a fixed trivialization) is already given and belongs to $W^{1,2}$ (if the bound by $\epsilon_0$ on the energy of $F$ is available, we may also assume more, by Uhlenbeck's result stated above, namely that one controls the $W^{1,2}$-norm of $A$ by the energy).\\

It can be seen from the formula of change of gauge that it is equivalent to estimate the gradient of the trivialization $g$ or the gradient of the connection $A$ in that gauge.\\

We define $f$ by iteration on $E$. The main steps are as follows (see the scheme \eqref{summary1}):
\begin{itemize}
 \item Uhlenbeck's theorem already gives a gauge, with an $L^4$-estimate of the gradient of the trivialization, in case the energy of $F$ is smaller than $\epsilon_0$. Instead of the $L^{4,\infty}$-estimate which we want, we get the stronger estimate in terms of the $L^4$-norm. The difficulty in our proof is to find an estimate without a priori assumptions on the $L^2$-smallness of $F$.
 \item Let $\rho_0$ be the largest scale at which no more than $\epsilon_0/2$ of $F$'s $L^2$-norm concentrates.
 \item In case $\rho_0\geq\bar\rho_0:=C\rho_{inj}(M)2^{-E/\epsilon_1}$ we iteratively extend our gauge on the simplexes of a triangulation where each simplex is well inside a ball of radius $\rho_{inj}(M)$. To do this we iteratively extend with $W^{1,3}$ estimates the change of gauge along the $3$-skeleton of the triangulation, then on each simplex we use Theorem B'' to extend inside that simplex. See Section \ref{itertriang}. The estimates depend only on $M^4$.
\item The other alternative is $\rho_0\leq \bar\rho_0$, or more explicitly
$$
\epsilon_1\log_{2}\frac{C\rho_{inj}}{\rho_0}\leq E.
$$
Then consider a point $x_0$ at which $|F|$ concentrates and look at the geodesic dyadic rings 
$$
R_k:=B(x_0, 2^{k+1}\rho_0)\setminus B(x_0,2^k\rho_0),\quad k\in\{0,\ldots,\lfloor\log_2(C\rho_{inj}/\rho_0)\rfloor\}.
$$
By pigeonhole principle, in one of these rings $D_{k_0}$ the curvature $F$ has energy less or equal than $\epsilon_1$. The parameter $\epsilon_1$ can be chosen, depending only on $\epsilon_0$, in such a way that this estimate of the energy ensures the existence of a small energy slice along a geodesic sphere of radius $t\sim 2^{k_0}\rho_0$. We then have extensions of the connections with curvatures of energy smaller than $E-\frac{\epsilon_0}{2}$. We use Lemma \ref{epsilon1}(finding good slices). To avoid subtleties about traces we will ensure that these two connections coincide on an open set. The choice of slice is described in Section \ref{choiceofslice}.
\item Then we separately trivialize these two connections using the iterative assumption that the $f$ as described in the claim of our theorem is already defined on $[0,E[$. By iterative assumption we then define $f(E)$ based on $f(E-\epsilon_0/2)$ and on the function $f_1$ which appears in Theorem B. The detailed bounds are given in Sections \ref{largeenext} and \ref{smallenext}.
\end{itemize}

\begin{equation}\label{summary1}
\scalebox{0.9}{\xymatrix@C=-15pt{
&&*+[F=:<3pt>]{\text{energy} =E}\ar[dl]\ar[dr]\\
&*+[F=:<3pt>]{\rho_0<\bar\rho_0}\ar[d]&&*+[F=:<3pt>]{\rho_0\geq\bar\rho_0}\ar[d]\\
&*+[F-:<3pt>]{\text{dyadic balls until }\sim\rho_{inj}}\ar[d]&&*++[F=]{\text{Extend gauge}}\\
&*+[F-:<3pt>]{\text{small energy slice at }\sim\rho_1}\ar[dr]\ar[dl]\\
*+[F-:<3pt>]{A_1,A_2\text{ of energy}\leq E-\frac{\epsilon_0}{2}}\ar[d]&&*+[F-:<3pt>]{A_1,A_2\text{ of energy}\leq \epsilon_0}\ar[d]\\
*++[F=]{\text{Iterate}}&&*++[F=]{\text{Extend gauge}}
}
}
\end{equation}

\subsection{Iterations based on a suitable triangulation}\label{itertriang}
Define, for $\epsilon_0$ as in Theorem \ref{uhlenbeckext}(Uhlenbeck gauge), the following radius:
\begin{equation}\label{rho0}
 \rho_0:=\inf\left\{\rho>0:\:\exists x_0\in M,\:\int_{B_\rho(x_0)}|F|^2=\frac{\epsilon_0}{2}\right\}.
\end{equation}
Denote 
$$
\bar\rho_0:=C\rho_{\op{inj}}(M)2^{-\frac{E}{\epsilon_1}},
$$
where $\rho_{\op{inj}}(M)$ is the injectivity radius of $M$ and the constant $\epsilon_1$ will be fixed later and depends only on the geometry of $M$ and on $\epsilon_0$. Fix then a triangulation on $M$ having in-radius $\gtrsim \bar\rho_0$ and size $\lesssim \bar\rho_0$, with implicit constants bounded by $4$. $C<1$ in the definition of $\bar \rho_0$ can be fixed now, so that each simplex of the triangulation is contained in a ball of radius $\rho_{\op{inj}}(M)/2$. In particular all $k$-simplexes of the triangulation are bi-Lipschitz equivalent to $\mathbb S^k$ with bi-Lipschitz constants which depend just on $k$.\\

Theorem \ref{uhlenbeckext}(Uhlenbeck gauge) gives a trivialization $\phi_i$ associated to each $4$-simplex $C_i$, such that the expression of $A$ in those coordinates 
\begin{equation}\label{u0}
A_i=\phi_i^{-1}d\phi_i + \phi_i^{-1}A\phi_i\text{  on  }C_i
\end{equation}
satisfies 
\begin{equation}\label{u1}
\|A_i\|_{W^{1,2}(C_i)}\leq C\|F\|_{L^2(C_i)}.
\end{equation}
If we call 
\begin{equation}\label{u2}
g_{ij}:=\phi_j^{-1}\phi_i
\end{equation}
then $g_{ij}g_{jk}=g_{ik}$, in particular we have $g_{ij}^{-1}=g_{ji}$; moreover
\begin{equation}\label{u3}
A_j=g_{ij}dg_{ji}+g_{ij}A_ig_{ji}\text{  on  }\partial C_i\cap\partial C_j.
\end{equation}
In particular, it follows from the above expression that $g_{ij}\in W^{1,3}(\partial C_i\cap\partial C_j, SU(2))$. We now state a lemma which will enable us to extend the gauge from one $4$-simplex to the next one.
\begin{lemma}[extension on a sphere]\label{exts3}
 Let $S_+^3$ be the upper hemisphere $\mathbb S^3\cap\{x_3\geq 0\}$. Then for any $g\in W^{1,3}(S_+^3,SU(2))$ there exists $\tilde g\in W^{1,3}(\mathbb S^3,SU(2))$ such that $\tilde g=g$ on $\mathbb S^3_+$ and 
$$
\|\nabla \tilde g\|_{L^3(\mathbb S^3)}\leq C\|\nabla g\|_{L^3(S_+^3)}.
$$
\end{lemma}
\begin{proof}
Up to enlarging $\mathbb S^3_-$ to a spherical cap of height $\leq 3/2$, we may assume that for a universal constant $C>0$
\begin{equation}\label{1}
\|g|_{\partial \mathbb S^3_-}\|_{W^{1,2}(\partial \mathbb S^3_-)}\leq C \|g\|_{W^{1,3}(\mathbb S^3_+)}.
\end{equation}
 We observe that $g|_{\partial S_+^3\simeq \mathbb S^2}\in W^{1,2}(\mathbb S^2, SU(2))$ and we want to extend this trace inside $B^3\simeq \mathbb S^3_-$ with a good norm estimate. We start with a harmonic extension (identifying $SU(2)\simeq\partial B^4$), namely
$$
\left\{\begin{array}{l}
        \Delta \hat g=0 \text{ on }B^3,\\
        \hat g = g\text{ on }\partial B^3.
       \end{array}
\right.
$$
Then we have by the usual elliptic estimates 
\begin{equation}\label{2}
\|\hat g\|_{W^{1,3}(\mathbb S^3_-)} \leq C \|g|_{\partial \mathbb S^3_-}\|_{W^{1,2}(\partial \mathbb S^3_-)}.
\end{equation}
We then observe that for $a\in B_{1/2}^4$ if $g_a$ is the radial projection of the values of $\hat g$ on the boundary with center $a$, then the following pointwise inequality holds (as in the projection trick of Section \ref{ssec:hardtlintrick})
\begin{equation}\label{3}
 |\nabla g_a|\leq C\frac{|\nabla\hat g|}{|\hat g -a |}.
\end{equation}
We also have
$$
\int_{a\in B_{1/2}^4}\int_{B^3}|\nabla g_a|^3\leq C\int_{B^3}|\nabla \hat g|^3.
$$
Therefore there exists $a\in B^4_{1/2}$ such that 
\begin{equation}\label{4}
 \|\nabla g_a\|_{L^3(B^3\simeq \mathbb S^3_-)}\leq C \|\nabla\hat g\|_{L^3(B^3\simeq \mathbb S^3_-)}.
\end{equation}
Combining the inequalities \eqref{1}, \eqref{2}, \eqref{3} and \eqref{4} we obtain the thesis for $\tilde g=g_a$ with $a$ as above.
\end{proof}
\begin{corollary}[iteration step]\label{nexttriangle}
 Suppose that on our $4$-manifold $M$ a connection $A$ is fixed and an Uhlenbeck gauge $\phi_j$ is defined on a $4$-simplex $C_j$, i.e. the estimate \eqref{u1} holds with the notation \eqref{u0}. Also suppose that a global gauge $\phi_I$ is defined on a finite union of simplexes $C_I:=\cup_{\alpha\in I}C_{i_\alpha}$ and that $\partial C_j\cap C_I^{(3)}$ (where $C_I^{(3)}$ is the simplicial $3$-skeleton of $C_I$) contains some, but not all, $3$-faces of $C_j$. It is then possible to extend the gauge change $g_{ij}$ defined in \eqref{u2} to $\tilde g_{ij}$ defined on the whole of $\partial C_j$ with a norm bound 
$$
\|\nabla \tilde g_{ij}\|_{L^3(\partial C_j)}\leq C\|\nabla g_{ij}\|_{L^3(\partial C_j\cap C_I^{(3)})},
$$
where $C$ depends only on $M$.
\end{corollary}
\begin{proof}
 $H:=(\partial C_j\setminus C_I^{(3)})_\delta$ is bi-Lipschitz to a ball for $\delta$ equal to $2/3$ the smallest in-radius of a face of $C_j$. Here $A_\delta$ is a $\delta$-neighborhood of $A$ inside $\partial C_j$. Also let $H':=(\partial C_j\setminus C_I^{(3)})_{2\delta}$. Note that the triple $(\partial C_j,H, H')$ is $C$-bi-Lipschitz equivalent to $(\mathbb S^3, \mathbb S^3_-, K)$ where $K$ is the spherical cap of height $3/4$ extending $\mathbb S^3_-$. We may then apply the construction of Lemma \ref{exts3}(extension on a sphere) and a bi-Lipschitz deformation, in order to ``fill the hole'' $H$ extending the gauge $g_{ij}$ with estimates. The bi-Lipschitz constant is bounded by the geometric constraints on our triangulation and is independent of $A$ and of $g_{ij}$.
\end{proof}

Given Lemma \ref{exts3}(extension on a sphere) and Corollary \ref{nexttriangle}(iteration step) we proceed iteratively on the triangulation as follows (the indices labeling the simplexes are re-defined during the whole procedure in a straightforward way):
\begin{itemize}
 \item Suppose that we already defined the gauge $\tilde\phi_{j-1}$ on a set of $j-1$ simplexes $C_1,\ldots,C_{j-1}$, whose union forms a connected set. \item Consider a new simplex $C_j$ extending such connected set. This choice of notation brings us directly under the hypothesis of Corollary \ref{nexttriangle}(iteration step) and thus we are able to extend $g_{ij}$ to $\tilde g_{ij}$ as in the corollary. 
\item We next apply Theorem B'' and extend $\tilde g_{ij}$ to a gauge change $h_{ij}$ defined inside $C_j$ and satisfying
 \begin{equation}\label{estexttri}
\|\nabla h_{ij}\|_{L^{(4,\infty)}(C_j)}\leq f(\|\nabla \tilde g_{ij}\|_{L^3(C_j)})\leq C_0,                                                                                                                     
\end{equation}
with $C_0$ depending only on universal constants and on $\epsilon_0$. The function $f$ is explicitly expressed in the statement of Theorem B''.
\item On $\cup_{i<j}C_i$ we keep $\tilde\phi_{j}=\tilde\phi_{j-1}$, while on $C_j$ we define $\tilde\phi_j=\phi_jh_{ij}$. 
\end{itemize}
We see that this construction gives for the local expression $\tilde A_j$ corresponding to the gauge $\tilde \phi_j$ the bound
$$
\|\tilde A_j\|_{L^{(4,\infty)}(C_j)}\lesssim \|A_j\|_{L^4(C_j)} +\|\nabla h_{ij}\|_{L^{(4,\infty)}(C_j)}\leq \epsilon_0 + C_0.
$$
Iterating this gauge extension strategy for all simplexes of a triangulation we would obtain a global gauge $\tilde A$ on the whole of $M$ such that
\begin{equation}\label{esttri}
 \|\tilde A\|_{L^{(4,\infty)}(M)}\leq C(\text{number of simplexes})(C_0+\epsilon_0)\leq C\frac{\op{Vol}(M)}{\bar\rho_0^4},
\end{equation}
since the volume of each simplex is $\gtrsim \bar\rho_0^4$. The above bound depends on the geometry of $M$ and on the energy $E$ of the curvature only. Note that the above reasoning works only as long as $\rho_0\lesssim\bar\rho_0$. As noted before, so far we have little control on $\rho_0$, in particular we have no bound from below. For this reason we next consider the case $\rho_0\geq\bar\rho_0$.

\subsection{Extending the connection with small curvature changes}\label{extsmallcurvch}
We now concentrate on proving the following lemma:
\begin{lemma}[finding good slices]\label{epsilon1}
There exists a constant $\epsilon_1$ with the following properties.
If $M$ is a fixed $4$-manifold with a $W^{1,2}$-connection $A$ and if $B_{2t}(x_0)\subset M$ is a geodesic ball with the estimate
$$
t\int_{\partial \tilde B_t}|F|^2\leq \epsilon_1
$$
then there exists $\hat A\in W^{1,2}(\wedge^1 M, su(2))$ such that $\hat A=A$ on $B_t$ and 
$$
\int_{M\setminus B_t}|F_{\hat A}|^2\leq C\epsilon_1
$$
with a constant $C$ depending only on $M$. In particular it is possible to ensure $C\epsilon_1<\frac{\epsilon_0}{4}$, with $\epsilon_0$ as in Theorem \ref{uhlenbeckext}(Uhlenbeck gauge).
\end{lemma}
\begin{proof}
 Up to a change of gauge which does not increase the norm, we may assume the Neumann condition
\begin{equation}\label{e11}
 \langle A,\nu\rangle \equiv 0 \text{ on }\partial B_t.
\end{equation}
This is obtained for example by minimizing $\|g^{-1} dg +g^{-1}Ag\|_{L^2(B_t)}$ among gauge functions $g\in W^{2,2}(B_t, SU(2))$.\\

We next extend $A$ to $B_{2t}\setminus B_t$ by 
$$
\tilde A:=\pi^*i^*_{\partial B_t}A, \text{ where }\pi(x)=t\frac{x}{|x|}\text{ and }i_{\partial B_t}\text{ is the inclusion.}
$$
Using the hypothesis and the facts that $i^*_{\partial B_t}$ acts on $F_A$ by just forgetting about some of its components and that $\pi$ is bi-Lipschitz, we obtain 
$$
\int_{B_{2t}\setminus B_t}\left|d\tilde A +\frac{1}{2}[\tilde A,\tilde A]\right|^2\leq C \epsilon_1.
$$
We can apply a change of gauge $g(\sigma)$ depending only on the angular variable $\sigma\in\partial B^4$ such that $$
d^*_{\partial B_t}A_g|_{\partial B_t}=0.
$$ 
This preserves the condition \eqref{e11} and also gives the following behavior as $s\to 0$:
$$
C\epsilon_1\geq \int_{B_s\cap\partial B_t}|dA_g +\tfrac{1}{2}[A_g,A_g]|^2\geq \int_{B_s\cap\partial B_t}|dA|^2 - o(s) \int_{B_s\cap\partial B_t}|\nabla A|^2.
$$
Therefore $A_g\in W^{1,2}(\wedge^1\partial B_t, su(2))$, $\tilde A_g\in W^{1,2}(\wedge^1B_{2t}\setminus B_t, su(2))$ and both $A_g, \tilde A_g$ satisfy \eqref{e11}. Therefore $\tilde A_g$ extends by $A_g$ in a neighborhood of $\partial B_t$, giving still a $W^{1,2}$-gauge. We observe that by Sobolev embedding
$$
\int_{\partial B_t}|[A,A]|^2\lesssim \left(\int_{\partial B_t}|\nabla A|^2\right)^2,
$$
and by Hodge decomposition and using $d^*_{\partial B_t}A=0$
$$
\int_{\partial B_t}|\nabla A|^2\lesssim \int_{\partial B_t}(|dA|^2+|d^*A|^2)\lesssim \int_{\partial B_t}|F_A|^2 + \left(\int_{\partial B_t}|\nabla A|^2\right)^2.
$$
The above inequality implies an inequality of the form $X\leq\epsilon_1 +X^2$ by our hypothesis and the gauge invariance of the curvature, with $X=\|\nabla A\|_{L^2(\partial B_t)}^2$.\\

We may thus assume that 
$$
t\int_{\partial B_t}|\nabla A|^2\leq Ct\int_{\partial B_t}|F|^2,
$$ 
which allows us to use a cutoff procedure, defining $\hat A:=\chi_tA$ for a smooth $[0,1]$-valued cutoff function $\chi_t$ such that $\chi_t\equiv 1$ on $B_t$ and $\chi_t\equiv 0$ outside $B_{2t}$. With this choice and the above estimate for $\nabla A$ we obtain 
$$
\int_{B_{2t}}|F_{\hat A}|^2\leq \int_{B_t}|F_A|^2+C\epsilon_1
$$
and we can extend $\hat A\equiv 0$ outside $B_{2t}$ obtaining the wanted estimate.
\end{proof}
\begin{rmk}
We will use the above lemma only in order to obtain a new connection with a controlled small energy, but the modification from $A$ to $\hat A$ will not be used otherwise: we will only be interested to change the gauge on the region where $A=\hat A$.
\end{rmk}
The above lemma is used to select a radius giving a slice with small energy concentration, and to make an induction on the energy.

\subsection{Cutting $M$ by a small energy slice}\label{choiceofslice}
Suppose for this subsection that we are in the case $\rho_0<\bar\rho_0$. We start by defining the following positive number $\rho_1$, which uses the same constant $C$ as in the definition of $\bar\rho_0$:
$$
\rho_1:=\left\{\begin{array}{ll}
\inf\left\{\rho\geq\rho_0:\:\int_{B_{2\rho}\setminus B_\rho}|F|^2\leq\frac{\epsilon_1}{4}\right\}&\text{ if this is }<C\rho_{\op{inj}}(M),\\
                C\rho_{\op{inj}}&\text{ else.}
               \end{array}
\right.
$$
Note that because of the hypothesis $\rho_0<\bar\rho_0$ and because of the choice of $\epsilon_1$, the $\rho_1$ is rather small, in such a way that $B_{2\rho_1}$ is bi-Lipschitz to $B_1$. Thus Lemma \ref{epsilon1}(finding good slices) applies. More precisely, we will apply the Lemma for two different radii $t_1\in[\rho_1, 5/4\rho_1], t_2\in[7/4\rho_1,2\rho_1]$. Chebychev's theorem implies the existence of $t_i,i=1,2$ such that
$$
t_i\int_{\partial B_{t_i}}|F|^2\leq \epsilon_1.
$$
We divide the proof into two cases, according to how large $\int_{M\setminus B_{2\rho_1}}|F|^2$ is with respect to $\epsilon_0$ from Theorem \ref{uhlenbeckext}(Uhlenbeck gauge).\\

\subsection{The case $\int_{M\setminus B_{2\rho_1}}|F|^2\geq\frac{\epsilon_0}{2}$}\label{largeenext}
In this case we split to the regions $B_{t_2}$ and $M\setminus B_{t_1}$ and do induction on the energy in order to find gauges satisfying our estimates on these two overlapping regions.\\
Lemma \ref{epsilon1}(finding good slices) gives extensions 
\begin{equation}\label{a1a2}
\left\{\begin{array}{ll}
        \hat A_1\equiv A\text{ on }B_{t_2}\text{ s.t. }&\int_M|F_{\hat A_1}|^2\leq \int_{B_{t_2}}|F_A|^2 + C\epsilon_1,\\
        \hat A_2\equiv A\text{ on }M\setminus B_{t_1}\text{ s.t. }&\int_M|F_{\hat A_2}|^2\leq \int_{B_{t_1}}|F_A|^2 + C\epsilon_1.
       \end{array}
\right.
\end{equation}
 In particular $\hat A_1,\hat A_2$ are equivalent on $B_{\frac{7}{4}\rho_1}\setminus B_{\frac{5}{4}\rho_1}$ and 
$$
\int |F_{\hat A_i}|^2\leq \int |F_A|^2-\frac{\epsilon_0}{4}.
$$
If we can find global gauges $g_i^\infty, i=1,2$ in which $\hat A_i$ have expressions $\hat A_i^\infty$ with $L^{(4,\infty)}$-bounds as in Theorem B, then it is enough to apply 
$$
g_{12}^\infty:=\left( g_1^\infty\right)^{-1} g_2^\infty
$$
on $R:=B_{\frac{7}{4}\rho_1}\setminus B_{\frac{5}{4}\rho_1}$ in order to obtain 
$$
A_2^\infty=g_{12}^\infty A_1^\infty\left(g_{12}^\infty\right)^{-1} + g_{12}^\infty d\left(g_{12}^\infty\right)^{-1}.
$$
This implies also 
$$
\|\nabla g_{12}^\infty\|_{L^{(4,\infty)}(R)}\leq f\left(E-\frac{\epsilon_0}{4}\right).
$$
Then there exists $t_3\in\left[\frac{5}{4}\rho_1,\frac{7}{4}\rho_1\right]$ such that 
$$
\int_{\partial B_{t_3}}|\nabla g_{12}^\infty|^3\leq f\left(E-\frac{\epsilon_0}{4}\right)
$$
and thus by Theorem B we can find a $W^{1,(4,\infty)}$-extension $h_{12}^\infty$ of $g_{12}^\infty$ to a map from $B_{t_3}$ to $SU(2)$. The estimate for $h_{12}^\infty$ is exactly as in Theorem B. Thus if we call $f_1$ the function of $\|\nabla\phi\|_{L^3}$ appearing Theorem B then
$$
\|\nabla h_{12}^\infty\|_{L^{(4,\infty)}(B_{t_3})}\leq f_1\left(f\left(E-\frac{\epsilon_0}{4}\right)\right)
$$
We then choose the following global gauge:
\begin{equation}\label{gg1}
 g^\infty:=\left\{\begin{array}{ll}
                   g_2^\infty&\text{ on }M^4\setminus B_{t_3},\\
                   h_{12}^\infty g_1^\infty&\text{ on }B_{t_3}.
                  \end{array}
\right.
\end{equation}
$\nabla g^\infty$ is then estimated by an universal constant times
$$
f_1(f(E-\epsilon_0/4))+f(E-\epsilon_0/4),
$$
which allows to define inductively $f(E)$.
\subsection{The case $\int_{M\setminus B_{2\rho_1}}|F|^2\leq\frac{\epsilon_0}{2}$}\label{smallenext}
In this case outside $B_{\rho_1}$ we apply directly Uhlenbeck's procedure, i.e. Theorem \ref{uhlenbeckext}(Uhlenbeck gauge), while on $B_{2\rho_1}$ we extend the so-obtained gauge via Theorem B''. If we call $A_1,A_2$ the so-obtained connections on $B_{2\rho_1},M\setminus B_{\rho_1}$ respectively, then 
$$
\exists t\in[\rho_1,2\rho_1]\text{ s.t. }\int_{\partial B_t}(|A_1|^3+|A_2|^3)\leq C(f_1(\epsilon_0)+\epsilon_0),
$$
thus as above the same bound is true also for the gradient of the change of gauge $\nabla g_{12}$. Then Theorem B gives the extension $h_{12}$ to a gauge in $W^{1,(4,\infty)}(B_t, SU(2))$. The estimate which we reach is
$$
\|\nabla h_{12}\|_{L^{4,\infty}(B_{t_3})}\leq f_1(C(f_1(\epsilon_0)+\epsilon_0)).
$$
We then choose 
\begin{equation}\label{gg2}
  g^\infty:=\left\{\begin{array}{ll}
                   g_2&\text{ on }M^4\setminus B_{t_3},\\
                   h_{12}g_1&\text{ on }B_{t_3}.
                  \end{array}
\right.
\end{equation}
This $g^\infty$ satisfies an estimate independent on $E$ and dependent only on $\epsilon_0$, again allowing to define $f(E)$ inductively.
\end{proof}

 \appendix

\section{Uhlenbeck small energy extension}\label{sec:Uhlenbeck}
\noindent
We now use the strategy which Uhlenbeck \cite{Uhl1} employed for the proof of controlled coulomb gauges under a small curvature requirement to prove Theorem \ref{smallenergyext}(Uhlenbeck analogue). We note that the analogy is in the method of proof more than in the result.\\

First observe that the following infimum is attained, as soon as the class on which we minimize is not empty (recall that $W^{1,2}(X, \mathbb S^3)=W^{1,2}(X, \mathbb R^4)\cap\{u: u(x)\in \mathbb S^3\text{ a.e.}\})$:
\begin{equation}\label{minimP}
 \inf\left\{\int_{B^4} |\nabla P|^2:\:P\in W^{1,2}(B^4, \mathbb S^3),\:P=P_0 \text{ on }\partial B^4\right\}.
\end{equation}
Indeed a minimizing sequence will have a $W^{1,2}$-weakly convergent subsequence, which will automatically also converge pointwise everywhere. In particular the constraint $u(x)\in \mathbb S^3$ a.e. is preserved. By weak lowersemicontinuity a minimizer exists, and by convexity it is unique. The minimizer $P$ verifies the following equation in the sense of distributions:
\begin{equation}\label{eqP}
\op{div}(P^{-1}\nabla P)=0.
\end{equation}
In the language of differential forms we can rewrite 
\begin{equation}\label{ellpiticeqP}
 d^*(P^{-1}d P)=0.
\end{equation}
This $P$ will be our extension inside the domain, and we will now prove some estimates which prove useful later.

\begin{lemma}[a priori estimates]\label{lem:aprioriest}
There exists $\epsilon>0$ with the following property. Let $P$ with $||P-I||_{W^{1,4}(B^4)}\leq \epsilon$ be an extension of $P_0\in W^{1,3}(\mathbb S^3,\mathbb S^3)$ which satisfies also \eqref{eqP}. We identify $\mathbb S^3$ with the Lie group $SU(2)$. Then there exists a constant $C_\epsilon$ such that 
\begin{equation}\label{aprioriest}
 ||P-I||_{W^{4/3,3}(B^4)}\leq C_\epsilon ||\nabla P_0||_{L^3(\mathbb S^3,\mathbb S^3)}.
\end{equation}
\end{lemma}

\begin{proof}
 We will start by a $L^2$-Hodge decomposition of $P^{-1}d P$: this $1$-form can be written in the form
\begin{equation}\label{hodgeP}
 P^{-1}d P=dU+d^*V,
\end{equation}
where a description of $V$ is as the unique minimizer of
$$
\min\left\{\int_{B^4}|d^*V - P^{-1}dP|^2,\:*V|_{\partial B^4}=0,\:dV=0\right\}.
$$
The existence of a minimizer follows easily by convexity as for \eqref{minimP}. The Euler-Lagrange equation is
\begin{equation*}
 \left\{\begin{array}{l}
         \Delta V=dd^*V=dP^{-1}\wedge dP,\\
         dV=0,\\
         *V=0.
        \end{array}
\right.
\end{equation*}
The fact that $\Delta V=(d^*d+dd^*)V$ coincides with $d d^*V$ is a consequence of the constraint $dV=0$. We claim that the following estimate holds:
\begin{equation}\label{estimateVbdry}
 ||\nabla V||_{L^3(\partial B^4)}\lesssim \epsilon ||P-I||_{W^{1,4}(B^4)}.
\end{equation}
To see this, observe that by elliptic, H\"older and Poincar\'e estimates (observe that $d(P^{-1})=P^{-1}dP\;P^{-1}$ and $P, P^{-1}\in L^\infty$ with norm equal to $1$):
\begin{eqnarray}
 ||\nabla V||_{W^{1,2}(B^4)}&\lesssim &||dP^{-1}\wedge dP||_{L^2(B^4)}\lesssim||d(P^{-1})||_{L^4(B^4)}||dP||_{L^4(B^4)}\nonumber\\
 &\lesssim&||dP||_{L^4(B^4)}||P^{-1}||_{L^\infty}^8||\nabla P||_{L^4(B^4)}\label{estVinside}\\
&\lesssim&\epsilon||P-I||_{W^{1,4}(B^4)}.\nonumber
\end{eqnarray}
Then we use the trace and Sobolev embedding inequalities:
$$
||V||_{L^p(\partial B^4)}\lesssim ||V||_{W^{1-\frac{1}{q}, q}(\partial B^4)}\lesssim||V||_{W^{1,q}(B^4)},
$$
where in general, in dimension $n$ large enough,
$$
p=\frac{qn}{n-\left(1-\frac{1}{q}\right)q}
$$
so that for $n=4,q=2$ we obtain $p=3$. Therefore we can concatenate the two last chains of inequalities and we obtain \eqref{estimateVbdry}.\\

Using the trace of the Hodge decomposition formula \eqref{hodgeP} on the boundary, we obtain from \eqref{estimateVbdry} that
\begin{equation}\label{UmenoP0}
||dU - P_0^{-1} dP_0||_{L^3(\partial B^4)}\lesssim \epsilon ||P-I||_{W^{1,4}(B^4)}.
\end{equation}
As for $V$, for $U$ we have the following equation:
$$
\Delta U= d^*dU = d^*(P^{-1}dP)=0.
$$
To justify the last passage recall \eqref{ellpiticeqP}.\\
We apply the elliptic estimates for $U$ to obtain:
\begin{equation}\label{boundaryestU}
||d U||_{W^{1/3,3}(B^4)}\lesssim  ||\nabla U||_{L^3(\partial B^4)},
\end{equation}
while the triangle inequality and the fact that $||P_0||_{L^\infty}=1$ give together with \eqref{UmenoP0}:
\begin{eqnarray}
 ||U||_{L^3(\partial B^4)}&\lesssim&||dU - P_0^{-1}dP_0||_{L^3(\partial B^4)} +||P_0^{-1}dP_0||_{L^3(\partial B^4)}\nonumber\\
&\lesssim&\epsilon||P-I||_{W^{1,4}(B^4)} + ||dP_0||_{L^3(\partial B^4)}.\label{estimateU}
\end{eqnarray}
We now use again \eqref{hodgeP}, the triangle inequality and the estimates \eqref{estVinside}, \eqref{boundaryestU},\eqref{estimateU}:
\begin{eqnarray}
 ||P^{-1}dP||_{W^{1/3,3}(B^4)}&\lesssim&||d^*V||_{W^{1/3,3}(B^4)} + ||dU||_{W^{1/3,3}(B^4)}\nonumber\\
  &\lesssim&\epsilon||P-I||_{W^{1,4}(B^4)} +||dP_0||_{L^3(\partial B^4)}.\label{estimatePDP}
\end{eqnarray}

We write $dP=P\;P^{-1}dP$ and observe that $P\in L^\infty\cap W^{1,4}$ since $\mathbb S^3$ is bounded, while $P^{-1}dP\in W^{1/3,3}$ from \eqref{estimatePDP}. We now use Lemma \ref{triebelest}  for the product $fg$ with $f=P, g=P^{-1}dP$ and we obtain
\begin{equation}\label{triebelP}
 \|d P\|_{W^{1/3,3}(B^4)}\lesssim \|P^{-1}dP|\|_{W^{1/3,3}}\left(\|P\|_{L^\infty}+\|P-I\|_{W^{1,4}(B^4)}\right).
\end{equation}
Note again that $\|P\|_{L^\infty}=1$ and deduce then from \eqref{estimatePDP}, \eqref{triebelest} and Poincar\'e inequality that
\begin{equation}\label{apriorifinal}
 \|P-I\|_{W^{4/3,3}(B^4)}\leq C\|d P_0\|_{L^3(\mathbb S^3)} + C\epsilon \|P-I\|_{W^{1,4}(B^4)}.
\end{equation}
Using the Sobolev inequality related to the continuous embedding $W^{4/3,3}(B^4)\to W^{1,4}(B^4)$ we can absorb the $\|P-I\|$-term to the left and we obtain the thesis.
\end{proof}

We are now ready for the proof of the small energy extension result of Theorem \ref{smallenergyext}. We restate the same result with a slight change of notation and more details.

\begin{theorem}[small energy extension]\label{smallenergyext2}
 There exist two constants $\delta>0,C>0$ with the following property. Suppose $Q\in W^{1,3}(\mathbb S^3,\mathbb S^3)$ such that $\|dQ\|_{L^3(\mathbb S^3)}\leq \delta$. Then there exists an extension $P\in W^{1,4}(B^4,\mathbb S^3)$ satisfying the following estimate:
$$
\|P-I\|_{W^{1,4}(B^4)}\leq C \|dQ\|_{L^3(\mathbb S^3)}.
$$
\end{theorem}

\begin{proof}
Define the following two sets:
\begin{equation}\label{G}
 \mathcal G_\epsilon^\alpha =\left\{Q\in W^{1,3+\alpha}(\mathbb S^3,SU(2)):\:\|\nabla Q\|_{L^3}\leq\epsilon\right\}
\end{equation}
\begin{equation}\label{F}
 \mathcal F_{\epsilon,C}^\alpha=\left\{\begin{array}{ll}Q\in\mathcal G_\epsilon^\alpha:&\:\exists P\in W^{1,4+\alpha}(B^4,SU(2)),\\
 &\left\{\begin{array}{ll}\op{div}(P^{-1}\nabla P)=0&\text{ on }B^4\\ P=Q&\text{ on }\partial B^4,\end{array}\right.\\
&\|P-I\|_{W^{1,4}(B^4)}\leq K\|\nabla Q\|_{L^3(\partial B^4)}\\
& \|P-I\|_{W^{1,4+\alpha}(B^4)}\leq C\|\nabla Q\|_{L^{3+\alpha}(\partial B^4)}                            
\end{array}\right\}.
\end{equation}
The constant $K>0$ will be fixed later In this language, the theorem states that a $P$ with estimates similar to the definition of $\mathcal F^0_{\epsilon,C}$ can be constructed to extend any $Q\in\mathcal G^0_\delta$ when $\delta$ is small enough. The strategy of the proof is to use the supercritical spaces $\mathcal G^\alpha_\epsilon,\alpha>0$ to approximate $\mathcal G^0_\epsilon$. We divide the proof in five steps, paralleling Uhlenbeck's paper \cite{Uhl1}.
\begin{itemize} 
\item \textbf{Claim 1:}\emph{ $\mathcal G_\epsilon^\alpha$ is connected for all $\epsilon,\alpha\geq 0$.}
 \item \textbf{Claim 2:}\emph{ $\mathcal F_{\epsilon,C}^\alpha$ is closed (in $\mathcal G_\epsilon^\alpha$) with respect to the $W^{1,3+\alpha}$-norm for $\alpha\geq 0$ and for any $C>0$.}
\item \textbf{Claim 3:}\emph{ For $\epsilon>0$ small enough and $\alpha>0$, there exists $C=C_\alpha$ such that the set $\mathcal F_{\epsilon,C}^\alpha$ is open in $\mathcal G_\epsilon^\alpha$ with respect to the $W^{1,3+\alpha}$-topology.}
\item \textbf{Claim 4:}\emph{ $\mathcal G^0_\epsilon$ is contained in the $W^{1,3}$-closure of $\cup_{\alpha>0}\mathcal G_{2\epsilon}^\alpha$.}
\end{itemize}
\emph{Proof of Claim 1.} This is straightforward since $\mathcal G_\epsilon^\alpha$ is actually convex.\\

\emph{Proof of Claim 2.} Consider a family $Q_j\in\mathcal F_{\epsilon,C}^\alpha$ with associated $P_j$ as in \eqref{F} which converge to $Q$ in $W^{1,3+\alpha}$. We can extract a weakly convergent subsequence of the $P_j$ and the estimate passes to the limit by weak lowersemicontinuity (and by convergence of the $Q_j$). Similarly, the equations pass to weak limits, since they are intended in the weak sense.\\

\emph{Ideas for Claim 3.} For the proof we need to study the behavior of solutions to the equation $\op{div}(P^{-1}\nabla P)=0$, which is regarded here as an equation $\mathcal N_\alpha(P)=0$, with $P$ close to the constant $I$ which is a zero of $\mathcal N_\alpha$. The equation considered is elliptic. The proof of the claim is thus done by linearization of $\mathcal N$ near $I$ and by implicit function theorem. Ellipticity of the equation translates into inconvertibility of this linearized operator. The estimate of the $W^{1,4}$-norm will follow from the a priori estimate of Lemma \ref{lem:aprioriest} once we choose for example $K\leq C_\epsilon/2$. See Lemma \ref{ift} for the complete proof.\\

\emph{Proof of Claim 4.} Consider $Q\in G^0_\epsilon$. By density arguments we find a sequence $Q_i\in C^\infty(\mathbb S^3, SU(2))$ such that $Q_i\to Q$ in $W^{1,3}(\mathbb S^3, SU(2))$. The density of smooth functions in the Sobolev space $W^{1,p}(X,Y)$ where $X,Y$  are smooth compact manifolds was studied in \cite{Bethuel}, \cite{Linsobolev}, and this density is always true for $p\geq\op{dim}(X)$; see the cited papers and the references therein for more general results. As in the cited proofs of the density, the case $p=\op{dim}(X)$ is obtain by a limiting procedure on $p\to(\op{dim}(X))^+$, which for us means that we may assume as well $Q_i\in \mathcal G^{\alpha_i}_{\epsilon_i}$, for some sequence $\alpha_i\to 0^+$. We note that the $L^3$-norm of a function $f$ can be obtained as 
$$\lim_{q\to 3^+}\|f\|_{L^q}$$
so in particular we may assume up to extracting a subsequence that $\epsilon_i\leq2\epsilon$.\\

\textbf{End of proof.} Consider $Q$ as in the statement of the theorem. In other words, $Q\in\mathcal G^0_\delta$. We use Claim 4 to approximate $Q$ in $W^{1,3}$-norm by $Q_i\in\mathcal G^{\alpha_i}_{2\delta}$ with $\alpha_i>0$. From the first three claims above it follows that there exist functions $P_i\in W^{1,4+\alpha_i}(B^4, SU(2))$ such that 
$$
\|P_i -I\|_{W^{1,4}(B^4)}\leq K\|dQ_i\|_{L^3(\mathbb S^3)}\leq 2K\delta.
$$
The $P_i$ have a weakly convergent subsequence whose limit $P$ satisfies
$$
\left\{\begin{array}{ll}\op{div}(P^{-1}\nabla P)=0&\text{ on }B^4\\ P=Q&\text{ on }\mathbb S^3\end{array}\right. \text{ and }\|P-I\|_{W^{1,4}(B^4)}\leq 2K\delta.
$$
We now use the a priori estimates, Lemma \ref{lem:aprioriest}. For this, we will choose $\delta>0$ such that $2K\delta\leq\epsilon$ for $\epsilon$ as in Lemma \ref{lem:aprioriest}. We can then apply that lemma and obtain that
$$
\|P-I\|_{W^{1,4}(B^4)}\leq c\|P-I\|_{W^{4/3,3}(B^4)}\leq cC_\epsilon\|Q\|_{L^3(\mathbb S^3)}.
$$
This concludes the proof.
\end{proof}

\begin{rmk}[Need for a priori estimates]
 In the proof of Claim 3 of the above proof we use the fact that for $\alpha>0$ we have the Sobolev inequality (valid on compact $3$-dimensional manifolds) $\| Q\|_{C^0}\leq c_\alpha \|Q\|_{W^{1,3+\alpha}}$. The dependence of the resulting constant $C_\alpha$  on $\alpha$ comes from this inequality, in particular $C_\alpha\to \infty$ for $\alpha\to 0^+$. The a priori estimate of Lemma \ref{lem:aprioriest} used in the last step of the proof is crucial precisely for this reason.
\end{rmk}

We now use the inverse function theorem for the operator $P\mapsto\op{div}(P^{-1}\nabla P)$.

\begin{lemma}\label{ift}
There exist $\epsilon>0, K>0$ such that for all $\alpha>0$ there exists $C_\alpha>0$ with the following properties.\\

Let $Q_0\in W^{1,3+\alpha}(\mathbb S^3, SU(2))$ and let $P_0\in W^{1,4+\alpha}(B^4, SU(2))$ be an extension of $Q_0$ which satisfies $\op{div}(P_0^{-1}\nabla P_0)=0$. If the following estimates hold:
\begin{eqnarray}
\|dQ_0\|_{W^{1,3}(\mathbb S^3)}&<&\epsilon,\label{w13}\\
 \|P_0-I\|_{W^{1,4}(B^4)}&\leq& K\|dQ_0\|_{W^{1,3}(\mathbb S^3)},\label{w14}\\
\|P_0-I\|_{W^{1,4+\alpha}(B^4)}&\leq& C_\alpha\|dQ_0\|_{W^{1,3+\alpha}(\mathbb S^3)},\label{w14a}
\end{eqnarray}
then for some $\delta>0$ depending on $Q_0$, for all $Q$ satisfying
\begin{equation}\label{deltaball}
 \|Q-Q_0\|_{W^{1,3+\alpha}(\mathbb S^3, SU(2))}<\delta,
\end{equation}
there exists an extension $P$ of $Q$ satisfying the same equation $\op{div}(P^{-1}\nabla P)=0$ and such that \eqref{w13}, \eqref{w14}, \eqref{w14a} hold with $P,Q$ in place of $P_0,Q_0$.
\end{lemma}
\begin{proof} 
We fix $Q$ satisfying \eqref{deltaball} and \eqref{w13}. The proof is divided in two parts:
\begin{itemize}
 \item \textbf{Claim 1:}\emph{ For $\delta n>0$ small enough and for $Q$ satisfying \eqref{deltaball} there exists an extension $P$ of $Q$ solving $\op{div}(P^{-1}\nabla P)=0$ and such that \eqref{w14a} holds.}
 \item \textbf{Claim 2:}\emph{ The function $P$ of Claim 1 satisfies \eqref{w14}.}
\end{itemize}
\emph{Proof of Claim 1.} 
First note that $V=\op{exp}^{-1}(Q_0^{-1}Q)$ is well defined for $\alpha>0$ because in that case we have an estimate of the form
$$
\|Q-Q_0\|_{W^{1,3+\alpha}}\geq c_\alpha\|Q-Q_0\|_{L^\infty}\Leftrightarrow \|Q_0^{-1}Q -I\|_{L^\infty}\leq \epsilon/c_\alpha
$$
and $\op{exp}^{-1}$ is well-defined in a neighborhood of the identity.\\

We consider the problem of extending $Q_0\op{exp}(V)$ inside $B^4$ to a function $P=P_0\op{exp}(U)$ satisfying \eqref{expeq}. Instead of considering the extension as a perturbation of $P_0$ only, we first extend $V$ to $\tilde V$ such that $\Delta \tilde V=0$ inside $B^4$.\\

We look for a $P$ of the form $P_0\op{exp}(\tilde V)\op{exp}(U)$. We thus consider the equation
\begin{equation}\label{expeq}
 \mathcal N(U,V):=d^*\left(\op{exp}(-U)\op{exp}(-\tilde V)P_0^{-1}d(P_0\op{exp}(\tilde V)\op{exp}(U))\right)=0.
\end{equation}
In order to solve \eqref{expeq} it is interesting to look at the operator
\begin{equation}\label{eq:isomorphism}
\mathcal N(V,U):W^{1,4+\alpha}_0(B^4, su(2))\to W^{-1,4+\alpha}(B^4,su(2)).
\end{equation}
We have to show that for $\delta>0$ small enough for each $Q$ satisfying $d^*(P^{-1}dP)=0$ (i.e. for each small enough $V$), there exists a unique $U$ such that $\mathcal N(V,U)=0$. Therefore it will be enough to show that $\partial \mathcal N/\partial U$ is an isomorphism between the two spaces above. It will be enough to restrict to the case where $V, U$ have norms $\leq C\delta$. Our estimates will prove that $\mathcal N(U,V)$ is $C^1$ near the couple $(0,0)$ and that $\partial \mathcal N/\partial U(0,0)$ is an isomorphism, given the existence of $\delta>0$ as wanted.\\

A simple calculation gives:
\begin{eqnarray*}
\frac{\partial \mathcal N}{\partial U}\cdot\eta&=&\left.\frac{\partial}{\partial t}\right|_{t=0}\mathcal N(U+t\eta,V)\\
&=&d^*d\eta- d^*\left[\eta,\op{exp}(-U)\op{exp}(-\tilde V)P_0^{-1}d(P_0\op{exp}(\tilde V))\op{exp}(U)\right]\\
&:=&\Delta\eta - L\eta.
\end{eqnarray*}
We observe that $d^*d=\Delta$ is an isomorphism between the spaces above, so it will be enough to show that for $U,\tilde V$ small enough in the $W^{1,4+\alpha}$-norm the commutator term $L\eta$ is just a small perturbation of $\Delta$ (with respect to the norms present in \eqref{eq:isomorphism}). First note that we can write
\begin{eqnarray*}
L\eta&=&[\nabla\eta,X] + [\eta,\op{div}X],\\
X&:=&\op{exp}(-U)\op{exp}(-\tilde V)P_0^{-1}d(P_0\op{exp}(\tilde V))\op{exp}(U)
\end{eqnarray*}
\textbf{Estimate for $[\nabla\eta,X]$.} First note that by the Sobolev, H\"older and triangle inequalities
$$
\|[\nabla\eta,X]\|W^{-1,4+\alpha}\lesssim \|[\nabla\eta,X]\|_{L^{p_\alpha}}\lesssim\|\nabla\eta\|_{L^{4+\alpha}}\|X\|_{L^4}.
$$
where 
$$
\frac{1}{p_\alpha}=\frac{1}{4+\alpha} +\frac{1}{4}.
$$
We then observe 
$$
X=\op{exp}(-U)\op{exp}(-\tilde V)P_0^{-1}d(P_0\tilde V)\op{exp}(\tilde V)\op{exp}(U)
$$
and note $|\op{exp}A|=1$ therefore
$$
\|X\|_{L^4}=\|d(P_0\tilde V)\|_{L^4}\lesssim\|dP_0\|_{L^4} +\|d\tilde V\|_{L^4}\lesssim\epsilon +\delta.
$$
We thus have the first wanted estimate
$$
\|[\nabla\eta,X]\|W^{-1,4+\alpha}\lesssim(\epsilon+\delta)\|\eta\|_{W^{1,4+\alpha}}.
$$
\textbf{Estimate for $[\eta,\op{div}X]$.} Here we start with
$$
\|[\eta,\op{div}X]\|_{W^{-1,4+\alpha}}\lesssim\|\eta\|_{L^\infty}\|\op{div}X\|_{L^{p_\alpha}}.
$$
Note that $\|\eta\|_{L^\infty}\lesssim\|\eta\|_{W^{1,4+\alpha}}$ by the Sobolev embedding. We start the computations for the second fact or above. Note
$$
\nabla (P_0\op{exp}\tilde V)=(\nabla P_0)\op{exp}\tilde V + P_0\nabla(\op{exp}\tilde V)
$$
and then expand:
\begin{eqnarray*}
 \op{div}X&=&\op{div}\Big[\op{exp}(-U)\op{exp}(-\tilde V)P_0^{-1}\nabla(P_0\op{exp}(\tilde V))\op{exp}(U)\Big]\\
 &=&\nabla\Big(\op{exp}(-U)\Big)\op{exp}(-\tilde V)P_0^{-1}\nabla(P_0\op{exp}(\tilde V))\op{exp}(U)\\
&&+\op{exp}(-U)\nabla\Big(\op{exp}(-\tilde V)\Big)P_0^{-1}\nabla(P_0\op{exp}(\tilde V))\op{exp}(U)\\
&&+\op{exp}(-U)\op{exp}(-\tilde V)\op{div}\Big(P_0^{-1}\nabla P_0\Big)\op{exp}(\tilde V))\op{exp}(U)\\
&&+\op{exp}(-U)\op{exp}(-\tilde V)P_0^{-1}P_0\op{div}\nabla\Big(\op{exp}(\tilde V)\Big)\op{exp}(U)\\
&&+\op{exp}(-U)\op{exp}(-\tilde V)P_0^{-1}\nabla P_0\nabla\Big(\op{exp}(\tilde V)\Big)\op{exp}(U)\\
&&+\op{exp}(-U)\op{exp}(-\tilde V)P_0^{-1}\nabla (P_0\op{exp}(\tilde V))\nabla\Big(\op{exp}(U)\Big)
\end{eqnarray*}
We have $\op{div}(P_0^{-1}\nabla P_0)=0$ and $\op{div}\nabla(\op{exp}(\tilde V))=0$ so two terms cancel. Note also the fact that $\|P_0^{-1}\nabla P_0\|_{L^4}\leq \|\nabla P_0\|_{L^4}\leq\epsilon$. Recall again that $|\op{exp}A|=1$ for all $A\in su(2)$. For estimating $\nabla(\op{exp}(\pm\tilde V))$ observe that $\tilde V$ satisfies a Dirichlet boundary value problem therefore we assumed the estimate $\|\tilde V\|_{W^{1, 4+\alpha}}\lesssim \delta$, and $\|U\|_{W^{1, 4+\alpha}}\lesssim \delta$ which by the smoothness of $\op{exp}$ imply $\|\nabla(\op{exp}(\pm\tilde V))\|_{L^{4+\alpha}}\lesssim \delta$ and $\|\nabla(\op{exp}(\pm U))\|_{L^{4+\alpha}}\lesssim \delta$. From all this it follows that we can estimate
\begin{eqnarray*}
\|\op{div}X\|_{L^{p_\alpha}}&\lesssim&\|\nabla(\op{exp}(-U))\|_{L^{4+\alpha}}\|\nabla(P_0\op{exp}\tilde V)\|_{L^4}\\
&&+\|\nabla(\op{exp}(-\tilde V))\|_{L^{4+\alpha}}\|\nabla(P_0\op{exp}\tilde V)\|_{L^4}\\
&&+\|\nabla P_0\|_{L^4}\|\nabla(\op{exp}(\tilde V))\|_{L^{4+\alpha}}\\
&&+\|\nabla(\op{exp}(U))\|_{L^{4+\alpha}}\|\nabla(P_0\op{exp}\tilde V)\|_{L^4}\\
&\lesssim&\delta \|\nabla(P_0\op{exp}\tilde V)\|_{L^4} +\epsilon\delta\\
&\lesssim&\delta(\epsilon +\delta).
\end{eqnarray*}
We thus again combine all the estimates and obtain the wanted smallness result
$$
\|[\eta,\op{div}X]\|_{W^{-1,4+\alpha}}\lesssim\delta(\epsilon+\delta)\|\eta\|_{W^{1,4+\alpha}}.
$$
\textbf{Step 3.} We now have that 
$$
\|L\eta\|_{W^{-1,4+\alpha}}\lesssim(\delta+1)(\epsilon+\delta)\|\eta\|_{W^{1,4+\alpha}} 
$$
while 
$$
\|\Delta\eta\|_{W^{-1,4+\alpha}}\gtrsim\|\eta\|_{W^{1,4+\alpha}}.
$$
Therefore for small enough $\epsilon,\delta$ we have also
$$
\|(\Delta-L)\eta\|_{W^{-1,4+\alpha}}\gtrsim\|\eta\|_{W^{1,4+\alpha}}.
$$
This concludes the proof.
\end{proof}

\section{A product estimate with only one bounded factor}\label{sec:product}
\begin{lemma}[cf. \cite{brezismir}]\label{triebelest}
 Let $\Omega$ be a smooth compact $4$-manifold. If $f\in W^{1/3,3}(\Omega)$ and $g\in W^{1,4}\cap L^\infty(\Omega)$ then we have the following estimate, with the implicit constant depending only on $\Omega$:
$$
\|fg\|_{W^{1/3,3}(\Omega)}\lesssim\|f\|_{W^{1/3,3}(\Omega)}\left(\|g\|_{L^\infty(\Omega)} +\|g\|_{W^{1,4}(\Omega)}\right)
$$
\end{lemma}
\begin{proof}
The estimates for the non-homogeneous part of the norms are trivial, so we concentrate on the homogeneous part.\\

We use the Littlewood-Paley decompositions $f=\sum_{j=0}^\infty f_j,g=\sum_{k=0}^\infty g_k$, and we recall that the $W^{s,p}$-norm is equivalent to the Triebel-Lizorkin $\dot F^1_{4,2}$-norm and the $W^{\theta,4}$-norm is equivalent to the $F^s_{p,2}$-norm, where in general the following definition holds
$$
||f||_{\dot F^s_{p,q}}=\left\|\left|2^{ks}f_k(x)\right|_{\ell^q}\right\|_{L^p}.
$$

We use different notations $\|\cdot\|, |\cdot|$ for the different norms just to facilitate the reading of formulas. As is usual in the theory of paraproducts, we estimate separately the following three contributions (where $g^k:=\sum_{i=0}^kg_k$ ad similarly for $f^k$)
$$
fg=\sum_i f_ig^{i-4} + \sum_{|k-l|<4}f_kg_l+ \sum_i f^{i-4}g_i:=I+II+III.
$$
The support of $\widehat{(f_ig^{i-4})}$ is included in $B_{2^{i+2}}\setminus B_{2^{i-2}}$ thus there holds
\begin{equation}\label{pp1}
 \|I\|_{W^{\frac{1}{3},3}}=\left\|\sum_if_i g^{i-4}\right\|_{W^{\frac{1}{3},3}}\sim\left[\int_\Omega\left(\sum_i 2^{\frac{2i}{3}} |f_i g^{i-4}|^2\right)^{\frac{3}{2}}\right]^{\frac{1}{3}}.
\end{equation}
and analogously for $III=\sum_if^{i-4}g_i$. Regarding the term $II$ we will estimate only $II':=\sum_i f_ig_i$ because the same estimate will apply also to the finitely many contributions of the form $\sum_if_ig_{i+l}$ with $0<|l|<4$.\\

We start with the most difficult term $III$. From above we have
\begin{eqnarray*}
 \|III\|_{W^{\frac{1}{3},3}}&\sim&\left[\int\left(\sum_i 2^{\frac{2i}{3}}|f^{i-4}g_i|^2\right)^{\frac{3}{2}}\right]^{\frac{1}{3}}\\
 &\leq&\left[\int\left(\sum_i 2^{-\frac{4i}{3}}|f^{i-4}|^2\right)^{\frac{3}{2}}\left(\sum_i 2^{2i}|g_i|^2\right)^{\frac{3}{2}}\right]^{\frac{1}{3}}\\
 &\leq&\left[\int\left(\sum_i 2^{-\frac{4i}{3}}|f^{i-4}|^2\right)^6\right]^{\frac{1}{12}}\left[\int\left(\sum_i 2^{2i}|g_i|^2\right)^2\right]^{\frac{1}{4}}\\
 &\leq&\|f\|_{W^{-\frac{2}{3},12}}\|g\|_{W^{1,4}}\\
 &\leq&\|f\|_{W^{\frac{1}{3},3}}\|g\|_{W^{1,4}}.
\end{eqnarray*}
For the term $I$ we have
\begin{eqnarray*}
 \|I\|_{W^{\frac{1}{3},3}}&\sim&\left[\int\left(\sum_i 2^{\frac{2i}{3}}|f_ig^{i-4}|^2\right)^{\frac{3}{2}}\right]^{\frac{1}{3}}\\
 &\lesssim&\|g\|_{L^\infty}\|f\|_{W^{\frac{1}{3},3}}
\end{eqnarray*}
because of the estimate $\|g^{i-4}\|_{L^\infty}\lesssim\|g\|_{L^\infty}$. Finally we estimate $II'$ as promised We prove it by duality, namely we prove that $II'$ is bounded as a linear functional on the unit ball of the dual $W^{-\frac{1}{3},\frac{3}{2}}$. Consider therefore $h$ in this ball. We note that the support of $\widehat{(f_ig_i)}$ is included in $B_{2^{i+2}}$ therefore some terms cancel
\begin{eqnarray*}
\int h\cdot II'&\sim&\sum_{k,i}\int h_kf_ig_i=\sum_{k\leq i+4}\int h_kf_if_j=\sum_i \int h^{i+4}f_ig_i\\
&\leq&\left|\sum_i\int 2^{-\frac{i}{3}}h^{i+4}2^{\frac{i}{3}}f_ig_i\right|\\
&\leq&\|g\|_{B_{\infty,\infty}^0}\int\left(\sum_i2^{-\frac{2i}{3}}|h^{i+4}|^2\right)^{\frac{1}{2}}\left(\sum_i2^{\frac{2i}{3}}|f_i|^2\right)^{\frac{1}{2}}\\
&\leq&\|g\|_{W^{1,4}}\|h\|_{W^{-\frac{1}{3},\frac{3}{2}}}\|f\|_{W^{\frac{1}{3},3}}
\end{eqnarray*}
The last estimate follows recalling that 
\[
 \|g\|_{B^0_{\infty,\infty}}:=\sup_i \|g_i\|_{L^\infty}
\]
and that in dimension $4$ we have continuous embeddings 
\[
 W^{1,4}\hookrightarrow\op{BMO}\hookrightarrow B_{\infty,\infty}^0.
\]
Summing up the different terms we conclude.
\end{proof}

\section{The M\"obius group of $B^4$}\label{sec:moeblemmas}
\noindent
We call the M\"obius group of $\mathbb R^n$ the group $M(\mathbb R^n)$ generated by all similarities and the inversion with respect to the unit sphere. Recall that a similarity is an affine map of the form 
$$
x\mapsto \lambda Kx +b\text{ with }\lambda>0, K\in O(n), b\in \mathbb R^n,
$$
and the inversion $i_{c,r}$ with respect to the sphere $\partial B(c,r)$ is the map
$$
x\mapsto c+r^2\frac{x-c}{|x-c|^2}.
$$
The formula $i_{c,r}=(r^2\:Id+c)\circ i_{0,1}\circ(Id-c)$ shows that all inversion belong to $M(\mathbb R^n)$.We use the following abridged notation: 
$$
x^*:=i_{1,0}(x)=x/|x|^2.
$$
The M\"obius group of $B^{n+1}$ is the subgroup $M(B^{n+1})$ of all transformations belonging to $M(\mathbb R^n)$ and which preserve $B^{n+1}$. Similarly we define the M\"obius group $M(\mathbb S^n)$ of the unit sphere $\mathbb S^n\subset\mathbb R^n$. The general form of an element $\gamma\in M(B^{n+1})$ is 
$$
\gamma= K\circ F_v\text{, with }K\in O(n),\:v\in B^4,\: F_v:= -v+(1-|v|^2)(x^*-v)^*.
$$
We use the following basic properties of the functions $F_v$ which can be found in \cite{alfhors}, Chap. 2:
\begin{lemma}\label{lem:basicpropfv}
\begin{itemize}
 \item There holds 
$$
|F_v|(x)=\frac{1-|v|^2}{[x,v]}
$$
where $[x,y]=|x||x^*-y|=|y||y^*-x|$.
\item $ F_v$ is conformal. We have $F_v^{-1}=F_{-v}$, $F_v(0)=-v$ and $F_v(v)=0$.
\item The conformal factor $|F_v'|(x)$ is explicitly computed as
$$
|F_v'|(x)=\frac{1-|v|^2}{1+|x|^2|v|^2-2x\cdot v}=\frac{|v^*|^2-1}{|x-v^*|^2}.
$$
\item The restriction $F_v|_{\mathbb S^3}$ belongs to $M(\mathbb S^3)$, in particular $ F_v|_{\mathbb S^3}$ is a conformal involution and 
$$
|(F_v|_{\mathbb S^3})'|(x)=\frac{1-|v|^2}{|x-v|^2}.
$$
\end{itemize}
\end{lemma}
The next lemma gives the estimate need in Lemma \ref{lem:tphibad} for the case when $v$ is close to $\partial B^4$:
\begin{lemma}\label{lem:morepropfv}
 Suppose that 
$$
\rho\leq\frac{1}{4}.
$$
Then on $F_v^{-1}(B_{1-\rho})$ the following estimate holds with a geometric constant $C$:
$$
\frac{h(v)}{C}\leq|F'_v|(x)\leq Ch(v).
$$
\end{lemma}
\begin{proof}
 We will calculate 
$$
\frac{\max\{|F_v'|(y): y\in F_v^{-1}(B_{1-\rho})\}}{\min\{|F_v'|(y'): y'\in F_v^{-1}(B_{1-\rho})\}}=\max\left\{\frac{|F_v'|(y)}{|F_v'|(y')}:\; y,y'\in F_v^{-1}(B_{1-\rho}) \right\}
$$
and we show that this quantity is bounded. The following equalities hold:
\begin{eqnarray*}
\max\left\{\frac{|F_v'|(x)}{|F_v'|(x')}:\; x,x'\in B_{1-\rho}\right\}&=&\max\left\{\frac{|F_{-v}'|(x)}{|F_{-v}'|(x')}:\; x,x'\in B_{1-\rho}\right\}\\
&=&\max\left\{\frac{|(F_v^{-1})'|(x)}{|(F_v^{-1})'|(x')}:\; x,x'\in B_{1-\rho}\right\}\\
&=&\min\left\{\frac{|F_v'|(F_v^{-1}(x'))}{|F_v'|(F_v^{-1}(x))}:\; x,x'\in B_{1-\rho}\right\}\\
&=&\min\left\{\frac{|F_v'|(y')}{|F_v'|(y)}:\; y,y'\in F_v^{-1}(B_{1-\rho})\right\}.
\end{eqnarray*}

From the formula of the previous lemma it follows that
$$
\nabla_x|F_v'|(x)=2\frac{|v^*|^2-1}{|v^*-x|^4}(v^*-x),
$$
therefore $|F_v'|$ achieves its extrema on $B_{1-\rho}$ at $\pm(1-\rho)\frac{v}{|v|}$. The maximum $M$ and the minimum $m$ of $|F_v'|$ satisfy
\begin{eqnarray*}
 M&=&\frac{1-|v|^2}{1+|v|^2(1-\rho)^2 - 2(1-\rho)|v|}=\frac{1-|v|^2}{(1-(1-\rho)|v|)^2},\\
  m&=&\frac{1-|v|^2}{1+|v|^2(1-\rho)^2 + 2(1-\rho)|v|}=\frac{1-|v|^2}{(1+(1-\rho)|v|)^2},\\,\\
 \frac{M}{m}&=&\left(\frac{1+(1-\rho)|v|}{1-(1-\rho)|v|}\right)^2\sim(1-(1-\rho)|v|)^{-2}\sim 1,
\end{eqnarray*}
which finishes the proof.
\end{proof}

 \end{document}